\documentclass[graybox]{svmult}

\usepackage[all]{xy}
\usepackage{mathptmx}       
\usepackage{helvet}         
\usepackage{courier}        
\usepackage{type1cm}        
\usepackage{makeidx}         
\usepackage{graphicx}        
\usepackage{multicol}        
\usepackage[bottom]{footmisc}

\usepackage{amsmath}
\usepackage{amssymb}

\usepackage{enumerate}

\xyoption{all} \numberwithin{equation}{section}
\xyoption{all} \numberwithin{theorem}{section}
\xyoption{all} \numberwithin{corollary}{section}
\xyoption{all} \numberwithin{proposition}{section}
\xyoption{all} \numberwithin{conjecture}{section}
\xyoption{all} \numberwithin{lemma}{section}
\xyoption{all} \numberwithin{definition}{section}
\xyoption{all} \numberwithin{remark}{section}
\xyoption{all} \numberwithin{example}{section}

\newcommand{\nc}{\newcommand}
\nc{\cH}{\mathcal{H}} \nc{\cA}{\mathcal{A}} \nc{\cG}{\mathcal{G}}
\nc{\cC}{\mathcal{C}}
\nc{\cO}{\mathcal{O}}
\nc{\cI}{\mathcal{I}}
\nc{\cB}{\mathcal{B}} \nc{\cY}{\mathcal{Y}} \nc{\cK}{\mathcal{K}}
\nc{\cX}{\mathcal{X}} \nc{\cS}{\mathcal{S}} \nc{\cE}{\mathcal{E}}
\nc{\cF}{\mathcal{F}} \nc{\cZ}{\mathcal{Z}} \nc{\cQ}{\mathcal{Q}}
\nc{\cN}{\mathcal{N}} \nc{\cP}{\mathcal{P}} \nc{\cL}{\mathcal{L}}
\nc{\cM}{\mathcal{M}} \nc{\cR}{\mathcal{R}} \nc{\cT}{\mathcal{T}}
\nc{\cW}{\mathcal{W}} \nc{\cU}{\mathcal{U}} \nc{\cD}{\mathcal{D}}
\nc{\cJ}{\mathcal{J}} \nc{\cV}{\mathcal{V}}
\nc{\fr}{{\rightarrow}}
\nc{\rd}{red.deg}
\newcommand{\Q}{\mathbb{Q}}

\newcommand{\pr}{\mathbb P}

\newcommand{\sym}{\mbox{\upshape{Sym}}}

\newcommand{\rk}{\mbox{\upshape{rank}}}

\newcommand{\surjarrow}{\rightarrow\!\!\!\!\rightarrow}
\newcommand{\kapp}{\mathbb C}

\makeindex             

\begin{document}

\title*{Stability conditions and positivity of invariants of fibrations}
\author{M.A. Barja
and L. Stoppino}
\institute{M. A. Barja \at Departament de Matem\`atica  Aplicada I, Universitat Polit\`ecnica de Catalunya, ETSEIB Avda. Diagonal, 08028 Barcelona (Spain) \email{Miguel.Angel.Barja@upc.edu}. Partially supported by MICINN-MTM2009-14163-C02-02/FEDER, MINECO-MTM2012-38122-C03-01 and by Generalitat de Catalunya 2005SGR00557
\and L. Stoppino \at Dipartimento di Scienza ed Alta Tecnologia, Universit\`a dell'Insubria, Via Valleggio 11  Como  (Italy) \email{lidia.stoppino@uninsubria.it}.
Partially supported by PRIN 2009 ``Spazi di moduli e Teoria di Lie'', by FAR 2011 (Uninsubria) and by G.N.S.A.G.A.--I.N.d.A.M.}
%

%
\maketitle

\abstract{We study three methods that prove the positivity of a natural numerical invariant associated to $1$-parameter families of polarized varieties. All these methods involve different stability conditions. In dimension 2 we prove that there is a natural connection between them, related to a yet another stability condition, the linear stability. Finally we make some speculations and prove new results in higher dimension.
}

\section*{Introduction}
The general topic of this paper regards how stability conditions in algebraic geometry imply positivity.
One of the first results in this direction is due to Hartshorne  \cite{HarAmple}: a $\mu$-semistable vector bundle of positive degree over a curve is ample.
Other seminal results  are Bogomolov Instability Theorem \cite{B2} and Miyaoka's Theorem on the nef cone of projective bundles over a curve \cite{miyaoka}. These theorems -not accidentally- are recalled and used in this paper (Theorem \ref{BIT} and Theorem \ref{teomiyaoka}).

\medskip
An important example of this kind of result is provided by the various proofs of  the so-called {\em slope inequality}
for a non-locally trivial relatively minimal fibred surface $f\colon S \longrightarrow B$, with general fibre $F$ of genus $g \geq 2$:
$$K_f^2 \geq 4 \frac{g-1}{g}\chi_f.$$

There are at least 3 different proofs of this result.
One is due  to Cornalba and Harris for the Deligne-Mumford non-hyperelliptic stable case \cite{CH} (generalized to the general case by the second author \cite{LS}), and uses the Hilbert stability of the canonical morphism of the general fibre of $f$. In \cite{bost} Bost proves a similar result assuming  Chow stability.
Although the proofs of Cornalba-Harris and Bost are different, the results are almost identical, being Chow and Hilbert stability very close (Remark \ref{Chow-Hilbert}).
Another proof of the slope inequality, due to  Xiao  \cite{X}, uses  the Clifford Theorem on the canonical system of the general fibre combined with the Harder-Narashiman filtration of the vector bundle $f_*\omega _f$.
A third approach has been introduced more recently by Moriwaki in \cite{Mor1}; this method uses the $\mu$-stability of the kernel of the relative evaluation map $f^*f_*\omega _f \longrightarrow \omega _f$ restricted on the general fibres.
In \cite{A-K} there is a good account of the last two proofs.
Miyaoka's Theorem is a key tool in the proof of Xiao, and Bogomolov Theorem is the main ingredient of Moriwaki's approach.
So we see at least two stabilities conditions involved in the proof of the slope inequality for fibred surfaces: Hilbert (or Chow) stability and $\mu$-stability.

\medskip
In this paper we study  these three methods in a general setting.  Firstly we present them  with arbitrary line bundles -instead of the relative canonical one- and in arbitrary dimension, when possible. Then we make a comparison between them, finding that in dimension 2 there is  a yet another stability condition, the {\em linear stability}, that connects them. Finally  we make some speculations about the higher dimensional case, and we prove a couple of new applications.

\medskip

Let us describe in more detail the contents of the paper.
We consider the following setting.
Let  $f\colon X \longrightarrow B$ a fibred variety,
 ${\cL}$ a line bundle on $X$, and let $\cG\subseteq f_*\cL$ be a subsheaf of rank $r$.
A great deal of the results presented in the paper  are in a more general setting, but let us assume here for the sake of simplicity that the general fibre of $\cG$ is generating
and that $\cL$ is nef.
Following \cite{CH}, we consider the number
$e(\cL, \cG)=rL^n-n \deg\cG (L_{|F})^{n-1}$, which is an invariant of the fibration (Remark \ref{class-invariant}).
We introduce the following notation (Definition \ref{def-f-positivity}): we say that $(\cL, \cG)$ is {\em $f$-positive} when $e(\cL, \cG)\geq 0$.
In the case $n=2$, choosing $\cL=\omega_f$, the slope inequality is equivalent to $f$-positivity of $(\omega_f, f_*\omega_f)$.

The structure of the paper is the following.
In Section \ref{risultati}, after giving the first definitions, we  make some useful computations via the Grothendieck-Riemann-Roch Theorem (Theorem \ref{conto} and Propositions \ref{conto1,5} and \ref{conto2}): the number $e(\cL, \cG)$ appears as the leading term of a polynomial expression  associated to the relative Noether morphism
$$\gamma_h\colon \sym^h \cG\longrightarrow f_*\cL^{\otimes h}, \mbox{ for }h\gg 0.$$
We then give a new elementary proof of a consequence of Miyaoka's result (Theorem \ref{nonfaschifo}): if $\cL$ is nef and $\cG$ is sheaf semistable, then $(\cL,\cG)$ is $f$-positive.
This is the first case we see where a stability condition  implies $f$-positivity.

In Section \ref{METHODS} we describe the three methods, adding here and there some new contribution.
As an illustration we re-prove  along the way the slope inequality for fibred surfaces via the three methods (Examples \ref{slopeCH}, \ref{slopeX}, \ref{slopeM}).
The neat idea would be to extend them so that they all give as an output $f$-positivity of the couple $(\cL, \cG)$, under some suitable assumptions. The Cornalba-Harris and Bost methods are originally stated in the general setting; we present them providing a slight generalization of the first one. They prove $f$-stability with the assumption that the fibre over general
$t\in B$ is Hilbert or Chow semistable together with the morphism defined by the fibre $G_t:=\cG\otimes \mathbb C(t)$ (Theorem \ref{corCH} and Theorem \ref{teobost}).

After discussing these methods, we make in \ref{risultati-hilbert-stability} a digression on some applications that are specific to the Cornalba-Harris method. In particular we give in Proposition \ref{hilb-finito-prop} a  bound on the canonical slope of the fibred surfaces such that the $k$-th Hilbert point of $(F,\omega_F)$ is semistable  for {\em fixed } $k$. This suggests a possible meaningful stratification of the moduli space of curves $\cM_g$.

The method of Xiao was extended in higher dimensions by Konno \cite{konnotrig} and Ohno \cite{ohno}. We give a general compact version (Propositon \ref{xiaoinequality}). Xiao's method does not provide in general $f$-positivity; it gives an inequality between the invariants $L^n$ and $\deg \cG$ that has to be interpreted case by case.

Moriwaki's method is described in \ref{MORIWAKI}. It only works in dimension 2, and it gives $f$-positivity if the restriction of  the kernel sheaf $\ker(f^*\cG\longrightarrow \cL)$ is $\mu$-semistable on the general fibres.
We also provide a new condition for $f$-positivity, independent from the one of the theorem of Moriwaki (Theorem \ref{f-positivity-high-stability}).

It is natural to try and make a comparison between these results, and  between their assumptions: in particular, in the case of fibred surfaces all the three methods work because the canonical system enjoys many different properties or is there a red thread binding the three approaches?
In Section \ref{thread} we study the $2$-dimensional case. It turns out that  there is a yet another stability concept, the {\em linear stability}, playing a central role in all three methods. Indeed, we observe the following:
\begin{itemize}
\item Section \ref{LS&CH}:  linear (semi-)stability can be assumed as hypotesis in the Cornalba Harris method, as it implies Chow (semi-)stability (Mumford and others).
\item Section \ref{LS&X}: linear (semi-)stability is the key assumptions that assures that the method of Xiao produces $f$-positivity.
\item Section \ref{LS&M}: linear (semi-)stability is implied by the stability assumption needed in Moriwaki's method and in a large class of cases is equivalent to it (Mistretta-Stoppino).
\end{itemize}
So the picture goes as as follows:
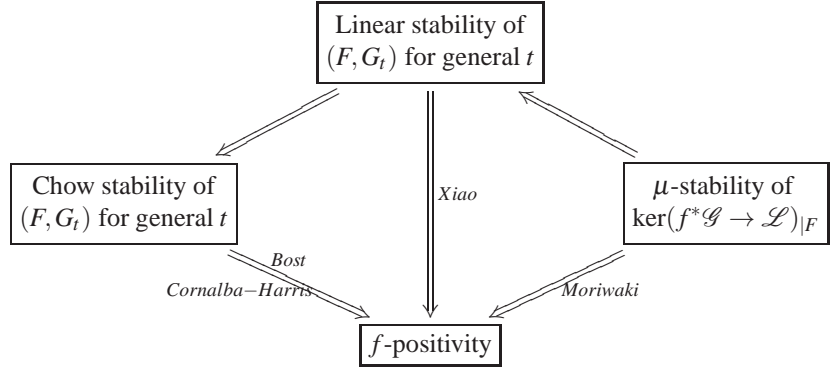
$$
\xymatrix{
& \framebox{$\begin{matrix} \mbox{Linear stability of} \\(F, G_t) \mbox{ for general }t\end{matrix}$} \ar@{=>}[ld]  \ar@{=>}^{Xiao}[dd] &\\
\mbox{\framebox{$\begin{matrix} \mbox{Chow stability of}\\ (F, G_t) \mbox{ for general }t\end{matrix}$}} \ar@{=>}^{Bost}_{Cornalba-Harris}[rd] &  &
\framebox{$\begin{matrix} \mbox{$\mu$-stability of }\\ \ker(f^*\cG\rightarrow \cL)_{|F}\end{matrix}$} \ar@{=>}[lu]  \ar@{=>}^{Moriwaki}[ld]\\
&\framebox{$f$-positivity} &
}
$$
In Section \ref{LS&X} we also prove some positivity results that can be proved via Xiao's method with weaker assumptions.

Finally in Section \ref{resultsHIGH} we consider the  higher dimensional case.
At this state of art, there is no hope to reproduce in higher dimension the beautiful connection between the three methods described for dimension 2.
First of all, the method of Moriwaki seemingly can not even be extended to dimension higher than $2$ (Remark \ref{bogomolov-high-dim}).
However, we provide some results regarding the other two methods.
Firstly we prove that the hypotesis of linear stability still implies a positivity result via Xiao's method (Proposition \ref{Xiao-high}).
In Section \ref{CH-high}, using known stability results, we can  prove new inequalities for families of abelian varieties  and of K3 surfaces via the Cornalba-Harris and Bost methods.
Moreover, we conjecture a higher-dimensional slope inequality to hold for fibred varieties whose relative canonical sheaf is relatively nef and ample (Conjecture \ref{conj-slope}).
We end the paper with an application of the (conjectured) slope inequality in higher dimension: using the techniques of Pardini \cite{Parda} it is possible to derive from the slope inequality a sharp Severi inequality $K_X^n\geq 2n!\chi(\omega_X)$ for $n$-dimensional varieties with maximal Albanese dimension (Proposition \ref{slope-severi}).
It is worth noticing that in \cite{Severi} the first author proves this Severi inequality, and Severi type inequalities for any nef line bundle, independently of such conjectured slope inequality.


\section{First results}\label{risultati}

\subsection{First definitions and motivation}\label{firstdefinitions}
We work over the complex field.
All varieties, unless differently specified, will be normal and projective.
Given a line bundle $\cL$ on a variety $X$, we call $L$ any (Cartier) divisor associated.
It is possible to develop the major part of the theory for reflexive sheaves associated to Weil $\mathbb{Q}$-Cartier divisors, but in order to avoid cumbersome arguments, we will stitch to this setting.

Let $X$ be a variety of dimension $n$, and  $B$ a smooth projective curve.
Let $f\colon X\longrightarrow B$ be a flat proper morphism with connected fibres.
Throughout the paper we shall call this data $f\colon X\longrightarrow B$ a {\em fibred variety}.

Let $\cL$ be a line bundle on $X$.
The pushforward  $f_*\cL$ is a torsion free coherent sheaf on the base $B$, hence it is locally free because $B$ is smooth 1-dimensional.
Let $\cG\subseteq f_*\cL$ be a subsheaf of rank $r$.
The sheaf $\cG$ defines a family of $r$-dimensional linear systems on the fibres of $f$,
$$G_t:=\cG\otimes \kapp (t)\subseteq H^0(F, {\cal L}_{|F}),$$
where $t\in B$ and $F=f^*(t)$.
Let us recall that the evaluation morphism 
$${\it ev} \colon f^*\cG\longrightarrow \cL$$
is surjective at every point of $X$  if and only if it induces a morphism $\varphi$ from $X$ to the relative projective bundle $\pr:={\mathbb{P}}_B({\cG})$ over $B$

$$
\xymatrix{
X\ar^{\varphi\quad\quad}[r] \ar[d]^f &
{\mathbb{P}}_B({\cG}):={\mathbb{P}}
\ar[dl]^\pi \\
B}
$$

\smallskip

\noindent such that $\cL=\varphi^*(\cO_\pr(1))$.
We will denote the surjectivity condition for $ev$ by saying that the sheaf $\cG$ is {\em generating} for $\cL$.
If ${\it ev}$ is only generically surjective, it defines a rational map $\varphi\colon X\dashrightarrow \pr$.
In this case, let $D$ be the unique  effective divisor such that $f^*\cG \longrightarrow  \cL( - D)$ is surjective in codimension 1. The divisor $D$ is called the {\em fixed locus} of $\cG$ in $X$.
Clearly the evaluation morphism $f^*\cG\longrightarrow \cL(-D)$ is surjective in codimension 1.

Moreover, by Hironaka's Theorem, there exist a desingularization $\nu\colon \widetilde X\longrightarrow X$ and a  morphism $\widetilde\varphi \colon \widetilde X\longrightarrow \pr$ such that  $\widetilde \varphi=\varphi\circ\nu$, and an effective $\nu$-exceptional  divisor  $E$ on  $\widetilde X$ such that
$$\widetilde\varphi^*(\cO_\pr(1))\cong \nu^*(\cL(-D))\otimes\cO_{\widetilde X}(-E).$$
See \cite[Lemma 1.1]{ohno} for a detailed proof of these facts. 
%
Define $\cM:= \widetilde\varphi^*\cO_\pr(1)\subseteq \nu^*\cL$; following \cite{ohno} we call this the {\em moving part} of the couple $(\cL, \cG)$, and we define the {\em fixed part} of $(\cL, \cG)$ on $\widetilde X$ to be $Z:=\nu^* (D)+E$.
Call $\widetilde f:= f\circ \nu$ the induced fibration.
Clearly the evaluation homomorphism $\widetilde f^*\cG\longrightarrow  \cM$ is surjective at every point of $\widetilde X$, i.e. $\cG$ is generating for $\cM$ on $\widetilde X$.

\begin{example}\label{moving-surfaces}
Let $f\colon S\longrightarrow B$ be a fibred surface, assuming for simplicity that $S$ is smooth. Let $\omega_f=\omega_S\otimes f^*\omega_B^{-1}$ be the relative dualizing sheaf of $f$.
Let $g$ be the (arithmetic) genus of the fibres.
The general fibres are smooth curves of genus $g$. Let us assume that $g\geq 2$: then the restriction of $\omega_f$ on the general fibres is ample. Hence the base divisor $D$ is vertical with respect to $f$.
Moreover, the line bundle $\omega_f$ has negative degree only on the $(-1)-$curves contained in the fibres.
So, all the vertical $(-1)-$curves of $S$ are contained in $D$.
It is possible to contract these curves preserving the fibration, and obtaining a unique {\em relatively minimal} fibration associated whose relative dualizing sheaf is $f$-nef.
However, there could still be a divisorial fixed locus, as we see now for the case of nodal fibrations.

Let us suppose that $f$ is a nodal fibration, i.e. that any fibre of $f$ is a reduced curve with only nodes as singularities.
We now describe explicitly the moving and the fixed part of $(\omega_f, f_*\omega_f)$.
Let us first recall the following simple result, that can be found in
\cite[Prop. 2.1.3]{Mor1}. If  $C$ is a nodal curve, the base locus of $\omega_C$ is given by all the disconnecting nodes and all the smooth rational components of $C$ that are attached to the rest of the fibre only by disconnecting nodes; following \cite{Mor1} we call these components of {\em socket type}.

The fixed locus of $(\omega_f, f_*\omega_f)$ is the union $D$ of all components of socket type.
Indeed, by what observed above the evaluation homomorphism $ev\colon f^*f_*\omega_f\longrightarrow \omega_f$ factors through $\omega_f(-D)$. On the other hand, it is easy to verify that the restriction of $\omega_f(-D)$ on any fibre is well defined except that on the disconnecting nodes not lying on components of socket type, so $f^*f_*\omega_f(-D) \longrightarrow \omega_f(-D)$ is surjective in codimension one.

Let  $\nu\colon \widetilde S\longrightarrow S$ be the blow up of all the base points of the map induced by $f_*\omega_f(-D)$; call $E$ the exceptional divisor, and $\widetilde f= f\circ \nu$ the induced fibration on $\widetilde S$.
Then we have that all the components of $E$ are of socket type for the corresponding fibre, and that the union of all the components of socket type of the fibres of $\widetilde f$ is $\widetilde D+E$, where $\widetilde D$ is the inverse image of $D$.
Thus $\widetilde D+E$ is the fixed part of $(\omega_{\widetilde f}, \widetilde f_*\omega_{\widetilde f})$, and the evaluation homomorphism 
$$\widetilde f^*\widetilde f_*\omega_{\widetilde f}(-\widetilde D-E) \longrightarrow \omega_{\widetilde f}(-\widetilde D-E)$$ 
is surjective at every point.
Noting that $\omega_{\widetilde f}\cong \nu^* (\omega_f)\otimes\cO_{\widetilde S}(E)$  (see for instance \cite[Chap.1, Theorem 9.1]{BHPVdV}),  we have that
$$ \omega_{\widetilde f}(-\widetilde D-E)\cong \nu^*(\omega_f)\otimes \cO_{\widetilde S}(-\widetilde D)\cong  \nu^*(\omega_f(-D))\otimes \cO_{\widetilde S}(-E).$$
So  the moving part of $(\omega_f, f_*\omega_f)$ is $\cM\cong \nu^*(\omega_f(-D))\otimes \cO_{\widetilde S}(-E)$.


\end{example}

\medskip

Let us now come to the definition of the main characters of the play.
%
\begin{definition}
With the above notation, define the  {\it Cornalba-Harris invariant}
$$e(\cL,\cG):=r L^n- n\deg \cG (L_{|F})^{n-1},$$
where  $L$ is a divisor such that ${\cal L}\cong {\cal O}_X(L)$, and $F$ is a general fibre.
\end{definition}

\begin{remark}\label{class-invariant}
The number $e(\cL,\cG)$ is indeed invariant by twists of line bundles from the base curve $B$.
Indeed, if $\cA$ is a line bundle on $B$ we have
 $$\rk (\cG\otimes \cA)=\rk \cG=r ,\quad \deg(\cG\otimes \cA)=\deg \cG+r \deg \cA,$$
$$(L +f^*A)^n=L^n+n  \deg \cA L_{|F}^{n-1}, \quad (\cL\otimes f^*\cA)_{|F}\cong \cL_{|F}.$$
It is therefore immediate to verify that $e(\cL\otimes f^*\cA,\cG\otimes \cA)=e(\cL,\cG)$.
\end{remark}

\begin{remark}\label{topselfint}
There is another significant incarnation of the C-H invariant: the number $r^{n-1}e(\cL,\cG)$
is the top self-intersection of the divisor  $r L-\deg\cG F$.
\end{remark}

Let us now consider again a fibred surface $f\colon S\longrightarrow B$ as in Example \ref{moving-surfaces}.
Let $g\geq 2$ be the genus of the fibres and $b$ the genus of the base curve $B$.
The main relative invariants for $f$ are
$K^2_f= K_S^2-8(b-1)(g-1)$ and  $\chi_f=\chi(\cO_S)-\chi(\cO_B)\chi(\cO_F)=\chi(\cO_S)-(g-1)(b-1)$.
By Leray's spectral sequence and Riemann-Roch one sees that $\chi_f=\deg f_*\omega_f$.
The {\em canonical slope} $s_f$ of the fibration is defined as the ratio between $K_f^2$ and $\chi_f$.
The  slope $s_f$ have been extensively studied in the literature (see \cite{X}, \cite{A-K}, \cite{BS1}).


%
%

In a more general setting, given a line bundle $\cL$ on $X$ and a subsheaf $\cG\subseteq f_*\cL$, one can consider, when possible,  the ratio between $L^{n}$ and $\deg\cG$, as follows.

\begin{definition}
With the same notation as above, let us suppose moreover that $\deg \cG>0$.
We define the {\it slope of the couple $({\cal L},\cG)$} as
$$s_f({\cal L},\cG):=\frac{L^n}{{\rm deg}{{\cal G}}}.$$
When $\cG=f_*\cL$, we shall use the notation $s_f(\cL)$.
\end{definition}
There is a rich literature about the search of lower bounds for the slope, in particular about the canonical one. The most general result is the following (see \cite{Barja3folds}).
\begin{proposition}
Assume that $\cL$ and $f_*{\cL}$ are nef. Then $s_f(\cL) \geq 1$.
\end{proposition}
This bound is attained by a projective bundle on $B$ and its tautological line bundle.

\begin{remark}
The slope is {\em not} invariant by twists of line bundles.
Indeed, let $F=f^*(t)$ be a general fibre, and $G_t:=\cG\otimes \kapp (t)\subseteq H^0(F, {\cal L}_{|F}).$
Attached to the triple $(f,{\cal G},{\cal L})$ a $\it natural$ ratio appears, which depends on the geometry of the triple $(F, G_t, {\cal L}_{|F})$.
Indeed, consider the line bundle ${\cal L}(kF)$ obtained by ``perturbing'' ${\cal L}$ with $kF$ for $k \in \mathbb{N}$, and the corresponding perturbed sheaf $ {\cal G}\otimes {\cal O}_B(kt)\subseteq f_*({\cal L}(kF))\cong f_*\cL\otimes \cO_B(kt)$. Then we have that
     $$s_f(\cL, \cG)(k):=s_f({\cal L}(kF),{\cal G}\otimes {\cal O}_B(kt))=\frac{(L+kF)^n}{\deg{\cal G}\otimes {\cal O}_B(kt)}=\frac{L^n + kn(L_{|F})^{n-1}}{{\deg}({\cal G})+k\, \rk {\cal G}}.$$
Hence
$$\lim_{k\rightarrow \infty}s_f(\cL, \cG)(k)= n \frac{(L_{|F})^{n-1}}{\rk {\cal G}}.$$
This asymptotic ratio is related to  $e(\cL, \cG)$ as follows; we have that
\begin{equation}\label{asymptotic}
s_f(\cL,\cG)\geq n \frac{(L_{|F})^{n-1}}{\rk {\cal G}}\iff e(\cL,\cG)\geq 0.
\end{equation}
The positivity of the Cornalba-Harris invariant thus coincides with this natural bound on $s_f(\cL, \cG)$.
\end{remark}

\begin{remark}\label{slopeinequality}
Let us consider inequality (\ref{asymptotic}) in the case of a fibred  surface   $f\colon S \longrightarrow B$ of genus $g\geq 1$. It becomes
$$K_f^2 \geq 2 \frac{\deg\omega_F}{\rk f_*\omega_f}\deg f_*\omega_f = 4 \frac{g-1}{g}\chi_f.$$
This bound is the famous {\em slope inequality} for fibred surfaces mentioned in the introduction.
It holds true for non locally trivial relatively minimal fibred surfaces of genus $g\geq 2$ (\cite{CH} and \cite{LS}, \cite{X}, \cite{Mor1}).

The case of surfaces allows us to single out some positivity conditions on the family that seem to be necessary in general.
\begin{itemize}
\item the genus $g$ of the fibration is $\geq 2$ $\iff$ $\omega_f$ is ample on the general fibres of $f$;
\item $f$ is non-locally trivial $\iff$ $\chi_f>0$;
\item $f$ is relatively minimal $\iff$ the divisor $K_f$ is nef (Arakelov).
\end{itemize}
In particular, if the fibration is not relatively minimal, the slope inequality is easily seen to be false.
We see that indeed in order to prove the positivity of $e(\cL, \cG)$ we will often need similar conditions, in particular the relative nefness of $\cL$.
In \ref{CH-high}  we conjecture and discuss a natural slope inequality in higher dimension.\end{remark}


By now we have seen how the condition of positivity of $e(\cL,\cG)$ is very natural and produces significant bounds for the geometry of the fibration.
We shall thus give a name to this phenomenon:
\begin{definition}\label{def-f-positivity}
The couple $(\cL,\cG)$ is said to be {\em $f$-positive} (resp. {\em strictly $f$-positive}) if $e(\cL,\cG)\geq 0$ (resp. $>0$).
\end{definition}



\subsection{Some intersection theoretic computations}\label{firstresults}

As above, let $f\colon X\longrightarrow B$ be a fibred variety over a curve $B$.
Let $\cL$ be a line bundle on $X$ and $\cG\subseteq f_*\cL$ a subsheaf of rank $r$.
Consider the natural morphism of sheaves
$$
\gamma_h \colon \sym^h \cG \longrightarrow f_*\cL^{\otimes h},
$$
for $h\geq 1$.
The fibres of this morphism on general $t\in B$ are just the multiplication maps
$$
\gamma_h \otimes \kapp(t) \colon \sym^h G_t =H^0(\pr^{r-1}, \cO_{\pr^{r-1}}(h)) \longrightarrow H^0(F, \cL_{|F}^{\otimes h}),
$$
where  $F=f^*(t)$ and $G_t=\cG\otimes\kapp(t)\subseteq H^0(F,\cL_{|F})$.
Call $\cG_h$ the image sheaf, and $\cK_h$ the kernel of $\gamma_h$.
If $ \cG$ is relatively ample then  for $h\gg 0$  we have that $\cG_h=f_*\cL^{\otimes h}$ and that $\cK_h$ is just $\cI_{X/\pr}(h)$, the ideal sheaf of the image of $X$ in the relative projective space $\pr$ twisted by $\cO_{\pr}(h)$.

\begin{remark}\label{immagine}
Suppose now that $\cG$ is generating.
Let $\overline X:=\varphi(X)\stackrel{j}{\hookrightarrow}  \pr$ be the image of $X$, let $\overline f\colon \overline X\longrightarrow B$ the induced fibration, and let  $\overline \cL=j^*(\cO_\pr(1))$. Then if $\alpha\colon X\rightarrow \overline X$ is the restriction of $\varphi$, we have that $\cL=\alpha^*\overline \cL$.
Clearly, for $h\gg0$ the sheaf $\cG_h$ coincides with $\overline f_*\overline\cL^{\otimes h}$, and $\cK_h $ with $\cI_{\overline X/\pr}(h)$.
\end{remark}

Let us recall that the {\em slope}\footnote{Unfortunately this crash of terminology seems unavoidable, as both the notations are well established.} of a vector bundle $\cF$ on a smooth curve $C$ is the following rational number $\mu(\cF)=\deg\cF /\rk(\cF)$.

\begin{remark}\label{upperbound}
Note that $f$-positivity is equivalent to an upper bound on the slope of the sheaf $\cG$, namely
$$\mu(\cG)\leq \frac{L^n}{n(L_{|F})^{n-1}}.$$
\end{remark}
We can now prove a  simple condition for $f$-positivity.

\begin{theorem}
Suppose that there exists an integer $m\geq 1$ such that
\begin{itemize}
\item[(i)]  the couple $(\cL^{\otimes m}, \cG_m)$ is $f$-positive;
\item[(ii)] $m \mu (\cG)\leq \mu(\cG_m).$
\end{itemize}
Then $(\cL, \cG)$ is $f$-positive.
\end{theorem}
\begin{proof}
Assumption $(i)$  tells us that
$$
\mu(\cG_m)\leq \frac{m L^{n}}{n L_{|F}^{n-1}}.
$$
which, combined with  $(ii)$, gives the desired inequality.
\end{proof}

We see below that the C-H class appears naturally as the leading term of the expression
$$r \deg \cG_h-h \deg \cG\rk \cG_h$$ when computed as a polynomial in $h$.
This produces the following condition for $f$-positivity in terms of the slope of $\cG$ and of the one of $\cG_h$.

\begin{theorem}\label{conto}
With the above notation,
suppose that the sheaf $\cG\subseteq f_*\cL$ is generating and such that the morphism $\varphi$ it induces is generically finite on its image.

Then the following implications hold
\begin{itemize}
\item[(1)] If $\mu(\cG_h)\geq h\mu(\cG)$ for infinitely many $h> 0$, then $(\cL,\cG)$ is $f$-positive.
\item[(2)] If $(\cL,\cG)$ is strictly $f$-positive, then $\mu(\cG_h)\geq h\mu(\cG)$  for $h\gg 0$.
\end{itemize}
\end{theorem}
\begin{proof}
As in Remark \ref{immagine}, let $\overline X:=\varphi(X)\stackrel{j}{\hookrightarrow}  \pr$ be the image of $X$, let $\overline f\colon \overline X\longrightarrow B$ be the induced fibration, and let  $\overline \cL=j^*(\cO_\pr(1))$.
As observed in the remark,  the sheaf $\cG_h$ coincides with $\overline f_*\overline \cL^{\otimes h}$ for $h\gg 0$.
By Grothendieck-Riemann-Roch theorem  we have that
$$\deg \cG_h=\deg{\overline f_* \overline \cL^{\otimes h}}=h^n\frac{(\overline L)^n}{n!}+  \sum_{i\geq 1}(-1)^{i+1}\deg R^i\overline f_*\overline \cL^{\otimes h} +\cO(h^{n-1}),$$
and that
$$\rk \cG_h=\rk \overline f_*\overline \cL^{\otimes h}=h^0(F, \overline \cL_{|F}^{\otimes h})=$$$$=h^{n-1}\frac{(\overline L_{|F})^{n-1}}{(n-1)!}+\sum_{i\geq 1}(-1)^ih^i(F,\overline \cL_{|F}^{\otimes h})+ \cO(h^{n-2}).$$
Moreover, $\cG$ is relatively very ample as a subsheaf of $\overline f_*\overline\cL$, and so by Serre's vanishing theorem
$\deg R^i\overline f_*\overline \cL^{\otimes h} =0 $ and $h^i(F, \overline \cL_{|F}^{\otimes h})=0$ for $h\gg0$, and $i\geq 1$.
By the assumption, the map $\alpha \colon X\longrightarrow \overline X$ is generically finite of degree say $d$. Hence
$$
L^n=(\alpha^*\overline L)^n=d (\overline L)^n \,\,\mbox{ and} \quad(\overline L_{|F})^{n-1}=(\alpha^*\overline L_{|F})^{n-1}=
d(\overline L_{|F})^{n-1}.
$$
Putting all together, we have
\begin{equation}\label{GRR}
\begin{array}{ll}
\displaystyle{ \rk \cG \deg \cG_h-h \deg \cG\rk \cG_h}& =\displaystyle{\frac{h^n}{d(n!)}\left(\rk \cG L^n - \deg \cG L_{|F}^{n-1}\right)+\cO(h^{n-1})=}\\
&\\
 &=\displaystyle{\frac{h^n}{d(n!)}e(\cL,\cG)+\cO(h^{n-1}).}\\
\end{array}
\end{equation}
So, if we have that $\mu(\cG_h)\geq h\mu(\cG)$ for infinitely many  $h>0$, then the leading term of $\rk \cG \deg \cG_h-h \deg \cG\rk \cG_h$ as a polynomial in $h$ must be non-negative (in  particular inequality $\mu(\cG_h)\geq \mu(\cG)$ is satisfied for $h\gg 0$).
Vice-versa, if the leading term is strictly positive, then $\mu(\cG_h)\geq h\mu(\cG)$ for $h\gg0$.
\end{proof}

\begin{remark}\label{terminedopo}
If  we have that $e(\cL, \cG) $ is zero, then of course we cannot conclude that
$$\rk \cG \deg \cG_h-h \deg \cG\rk \cG_h\geq 0 \,\,\mbox{ for } \,\,h\gg 0.$$
However, we can in this case consider the term in $h^{n-1}$, which is
$$\frac{h^{n-1}}{(n-1)!}\left( (n-1)\deg \cG L_{|F}^{n-2}K_F- L^{n-1}K_f\rk \cG\right).$$
Using the equality $\rk \cG L^n =n \deg \cG L_{|F}^{n-1}$, this term becomes
$$
\frac{r h^{n-1}}{(n-1)!}\left( \frac{n-1}{n}\frac{L_{|F}^{n-2}K_F}{L_{|F}^{n-1}}L^n- L^{n-1}K_f\right).
$$
Note that in case $\cL=\omega_f$ we obtain $-\frac{1}{n} K_f^n$, so that we can observe that
if $K_f^n>0$ and $\omega_f$ is ample on the general fibres, then if $\mu(\cG_h)\geq h\mu(\cG)$  for infinitely many $h> 0$, $(\omega_f,f_*\omega_f)$ is  strictly $f$-positive.
\end{remark}

\begin{remark}\label{bah}
We can observe the following.
Consider the function $\psi(h):=\mu(\cG_h)/h$, and assume the same hypotesis as Theorem \ref{conto}.
Then, by the very same computations contained in the proof of Theorem \ref{conto}, we see that
$$\
\lim_{h\to\infty}\psi(h)= \frac{L^n}{n (L_{|F}^{n-1})}.
$$
Moreover observe that, for any $h\geq 1$
$$
(\cL^{\otimes h}, \cG_h) \mbox{ is $f$-positive }\iff \psi(h)\leq \frac{L^n}{n (L_{|F}^{n-1})}.
$$
Theorem \ref{conto} can thus be rephrased as the following behavior of the function $\psi$.
\begin{itemize}
\item[(1)] If $\psi (h)\geq \psi (1)$ for infinitely many $h$, then $\psi(1)\leq L^n/(nL_{|F}^{n-1})$.
\item[(2)] If $\psi(1)<  L^n/(nL_{|F}^{n-1})$, then $\psi (h)\geq \psi (1)$  for $h\gg 0$.
\end{itemize}

\end{remark}

\medskip

We state now a couple of results along the lines of Theorem \ref{conto}, when we  weaken as much as possible the assumptions needed in order to obtain $f$-positivity.

\begin{proposition}\label{conto1,5}
With the same notation as above, suppose that  the line bundle $\cL$ is nef on $X$ and that the base locus of $\cG$ is concentrated on fibres.

If $\mu(\cG_h)\geq h\mu(\cG)$ for infinitely many $h$, then $(\cL,\cG)$ is $f$-positive.
\end{proposition}
\begin{proof}
If the map $\varphi$ induced by $\cG$ is not generically finite on its image then $e(\cL, \cG)=0$, hence $f$-positivity is trivially satisfied. If on the contrary $\varphi$ is finite on its image, we can apply Theorem \ref{conto} using, instead of $\cL$, the moving part of $(\cL, \cG)$
$$\cM=\nu^*(\cL(-D))\otimes\cO_{\widetilde X}(-E),$$
where we follow the notation of Section \ref{risultati}.
Let $M$ be a divisor associated to $\cM$.
By Theorem \ref{conto}, we have that the assumption $\mu(\cG_h)\geq h\mu(\cG)$ for $h\gg 0$ implies that $(\cM, \cG)$ is $f$-positive, so that $M^n\geq n\mu(\cG)(M_{|F})^{n-1}$.
By the assumption on the base locus of $\cG$, we have that $M_{|F}\sim L_{|F}$. Moreover, as $\cL$ and $\cM$ are nef and $\cM$ is $\cL$ minus an effective divisor, we have that $L^n\geq M^n$.
Summing up, we have $L^n-n\mu(\cG)L_{|F}^{n-1}\geq M^n-n\mu(\cG)(M_{|F})^{n-1}\geq 0$, and so we are done.
\end{proof}

\begin{remark}\label{nef-rel-nef}
It is worth noticing that in the statement of Proposition \ref{conto1,5} above, we could replace the assumption of $\cL$ being nef with $\cL$ being {\em relatively} nef.
Indeed, as $e(\cL, \cG)$ is invariant by twists with pullback of line bundles on the base (Remark \ref{class-invariant}), we can always replace a relatively nef line bundle with a nef one, by twisting with the pullback of a sufficiently ample line bundle on $B$.
\end{remark}

\begin{proposition}\label{conto2}
With the same notation as above, suppose that

($\star$) for $h\gg 0$ and $i\geq 1$
 $\deg R^if_*\cL^{\otimes h}=\cO(h^{n-1})$
and   $h^i(F, \cL_{|F}^{\otimes h})=\cO(h^{n-2})$.

Suppose moreover that one of the following conditions hold
\begin{itemize}
\item[(a) ]  the sheaf $\cG\subseteq f_*\cL$ is normally generated for general $t\in B$;
\item[(b) ] the sheaf $f_*\cL^{\otimes h}$ is nef
for $h\gg 0$.
\end{itemize}
Then if $\mu(\cG_h)\geq h\mu(\cG)$  for infinitely many $h>0$, then $(\cL,\cG)$ is $f$-positive.
\end{proposition}
\begin{proof}
Suppose that condition $(a)$ holds: then for $h\gg 0$ the sheaf $\cG_h$ generically coincides  with  (and is contained in) $f_*\cL^{\otimes h}$. Hence, as we are on a smooth curve, $\deg\cG_h\leq \deg  f_* \cL^{\otimes h}$ for $h\gg 0$.
The same inequality holds true if condition $(b)$ is satisfied.

By Grothendieck-Riemann-Roch theorem as in Theorem \ref{conto} we have that
$$\deg{ f_* \cL^{\otimes h}}=h^n\frac{L^n}{n!}+  \sum_{i\geq 1}(-1)^{i+1}\deg R^i f_* \cL^{\otimes h} +\cO(h^{n-1}),$$
$$\rk  f_* \cL^{\otimes h}=h^0(F, \cL_{|F}^{\otimes h})=h^{n-1}\frac{(L_{|F})^{n-1}}{(n-1)!}+\sum_{i\geq 1}(-1)^ih^i(F,\cL_{|F}^{\otimes h})+ \cO(h^{n-2}).$$
Putting all together and using assumption $(\star)$ we have
\begin{equation}
\begin{array}{ll}
\displaystyle{ \rk \cG \deg \cG_h-h \deg \cG\rk \cG_h}& \geq \displaystyle{\frac{h^n}{n!}\left(\rk \cG L^n - \deg \cG L_{|F}^{n-1}\right)+\cO(h^{n-1})=}\\
&\\
 &=\displaystyle{\frac{h^n}{n!}e(\cL,\cG)+\cO(h^{n-1}),}\\
\end{array}
\end{equation}
and the conclusion follows as in the above theorem.
\end{proof}

\begin{remark}
Note that if we drop assumption $(\star)$, we still obtain an inequality, involving a correction term due to the higher direct image sheaves.

The results above are generalizations of a computation contained in the proof of the main theorem of \cite{CH} (see also \cite{LS} and \cite[sec.2]{BS2}), where it is treated the case where the general fibre of $\cG$ is very ample.
\end{remark}


\subsection{Stability and $f$-positivity: first results}

Let us recall that a vector bundle $\cF$ over a smooth curve $B$ is said to be {\em $\mu$-stable (resp. $\mu$-semistable)} if for any proper subbundle $\cS\subset \cF$ we have $\mu(\cS)<\mu(\cF)$ (resp. $\leq$).
This is equivalent to asking that for any quotient bundle $\cF\surjarrow \cQ$ we have $\mu(\cQ)>\mu(\cF)$ (resp. $\geq $).

Let us now consider as usual a fibred variety $f\colon X\longrightarrow B$ over a curve $B$.
Let $\cL$ be a line bundle on $X$ and $\cG\subseteq f_*\cL$ a generating  subsheaf of rank $r$.
We see here that $\mu$-semistability of $\cG$ implies $f$-positivity. This is the first case we encounter where a stability condition implies the positivity of the C-H invariant.
However, $\mu$-semistability on the base is quite a restrictive condition to ask (see Remark \ref{stab-serio}).
In  \ref{XIAO}, we will see a method due to Xiao that uses vector bundle techniques on $\cG$ to prove some positivity results, but does not need to assume $\mu$-semistability.
However, we will see in \ref{LS&X} that, in order to give $f$-positivity as a result, Xiao's method needs another stability condition on the general fibres, the so-called linear stability.

\medskip

We will need the following simple remark.

\begin{remark}
Let $\cF$ be a vector bundle of rank $r$ on a smooth curve $B$.
Observe that, if $h$ is  any integer $\geq 1$, we have the following equalities:
$$
\deg (\sym^h\cF)=\binom{h+r-1}{r}\deg \cF,\quad
\rk  (\sym^h\cF)=\binom{h+r-1}{r-1}.
$$
We thus easily deduce the following.
\begin{equation}\label{sym}
\mu(\sym^h\cF)=h\mu(\cF).
\end{equation}
\end{remark}

\begin{theorem}\label{nonfaschifo}
With the notation above, let us suppose that $\cG$ is generating, or that the assumptions of Proposition \ref{conto1,5}  or of Proposition \ref{conto2} hold.
Then the following holds: if the sheaf $\cG$ is $\mu$-semistable,
then $(\cL,\cG)$ is $f$-positive.
\end{theorem}
\begin{proof}
If $\cG$ is $\mu$-semistable then $\sym^h\cG$ is $\mu$-semistable for any $h$, so that we have that the inequality $\mu (\sym^h \cG)\geq \mu(\cG_h)$ is satisfied.
But $\mu (\sym^h \cG)=h\mu(\cG)$ by formula (\ref{sym}) above.

Then if  the conditions in Proposition \ref{conto1,5} or in Proposition \ref{conto2} are satisfied, we are done.

Let us now suppose that $\cG$ is generating. If the morphism $\varphi$ it induces is not generically finite on its image then $e(\cL,\cG)=0$.
If on the contrary $\varphi$ is generically finite on its image, by what we have seen above, we are in the conditions to apply Theorem \ref{conto}.
\end{proof}

\begin{remark}\label{stab-serio}
From the above argument, we see that the $\mu$-stability of $\cG$ is much more than we need to prove $f$-positivity: indeed, in order to assure $f$-positivity, we just need that for infinitely many $h>0 $ the sheaf $\cG_h$ is not destabilizing for $\sym^h\cG$, and this  condition is almost 
necessary (Proposition \ref{conto}). The condition of $\mu$-stability of  $\sym^h\cG$ implies instead that this sheaf does not have {\em any} destabilizing quotient.

Indeed, it seems that the $\mu$-stability of $\cG$ is an extremely restrictive condition to ask.
In order to illustrate this, consider any variety fibred over $\pr^1$, and consider the relative canonical sheaf $\omega_f$.
 If the sheaf $f_*\omega_f$ is $\mu$-semistable, then necessarily its rank has to divide its degree, so that
$h^0(F,K_F)$ necessarily divides $\deg f_*\omega_f$. Any fibred variety violating this numerical condition cannot have $f_*\omega_f$  $\mu$-semistable.
Moreover, let us recall Fujita's decomposition theorem for the pushforward of the relative canonical sheaf.  Given a fibration $f\colon X\longrightarrow B$, we have that
\begin{equation}
f_*\omega_f=\cA\oplus (\oplus^{q_f}\cO_B),
\end{equation}
where $q_f:= h^1(B,f_*\omega_X)$, and $H^0(B, \cA^*)=0$.
From this result we see that $f_*\omega_f$ fails to be semistable as soon as $q_f>0$. For instance, for any fibred surface $f\colon S\longrightarrow B$ with $q(S)>b$, the pushforward of the relative canonical sheaf needs to be $\mu$-unstable.
See \cite{X} (in particular Theorem 3) for some related results.
\end{remark}

\medskip

A weaker version of Theorem \ref{nonfaschifo} can be proved as a corollary of a beautiful result due to Miyaoka, as we see below.

\medskip

Let us first define the setting of Miyaoka's Theorem.
Let $\cF$ be a vector bundle over a smooth curve $B$. Let $\pi\colon \pr:=\pr_B(\cF) \longrightarrow B$ be the relative projective bundle, and let $H$ be a tautological divisor on $\pr$, i.e. $\cO_\pr(H)\cong \cO_\pr(1)$, and let $\Sigma $ be a general fibre of $\pi$.

\begin{theorem}[Miyaoka \cite{miyaoka}]\label{teomiyaoka}
Using the above notations, the sheaf $\cF$ is $\mu$-semistable if and only if the $\Q$-divisor
$$H-\mu(\cF)\Sigma $$ is nef.
\end{theorem}

\medskip

Applying Theorem \ref{teomiyaoka} to our situation we can deduce the following
\begin{corollary}\label{cormiyaoka}
Let $\cL$ be a nef line bundle, $f\colon  X\longrightarrow B$ a fibration and  $\cG\subseteq f_*\cL$  has base locus vertical with respect to $f$.
If the sheaf $\cG$ is $\mu$-semistable, then the couple $(\cL,\cG)$ is $f$-positive.
\end{corollary}
\begin{proof}
With the notations of Section \ref{risultati}, let  us observe that
$$\widetilde \varphi^*(H-\mu(\cG)\Sigma)=\nu^*(L-D)-E-\mu(\cG) F.$$
Recalling that $\widetilde \varphi$ is a morphism, by  Theorem \ref{teomiyaoka}  the divisor $\nu^*(L-D)-E-\mu(\cG) F$ is nef.
This divisor therefore has non-negative top self-intersection, and so the result follows using the same computations of Proposition \ref{conto1,5} and Remark \ref{topselfint}.
\end{proof}

\section{The three methods}\label{METHODS}



\subsection{Cornalba-Harris and Bost: Hilbert and Chow stability}\label{C&H}

We now present the method of Cornalba and Harris \cite{CH}, in the generalized setting introduced in \cite{LS}.
Let us start with  a definition. Let $X$ be a variety,
with a linear system $V\subseteq H^0(X,\cD)$,
for some line bundle  $\cD$ on $X$.
Fix $h\geq 1$ and call $G_h$ the image of the natural homomorphism
$$\sym^hV\stackrel{\varphi_h}{-\!\!\!\longrightarrow}H^0(X, \cD^{\otimes h}).$$
Set $N_h=\dim G_h$ and take exterior powers
\begin{equation}\label{omo}
\bigwedge^{N_h}\sym^hV\stackrel{\wedge ^{N_h}\varphi_h}{-\!\!\!\longrightarrow}
\bigwedge^{N_h}G_h=\det G_h.
\end{equation}
The map $\wedge ^{N_h}\varphi_h$ defines uniquely an element $\left[\wedge ^{N_h}\varphi_h\right]\in\pr (\wedge^{N_h}\sym^hV^{\vee})$ which we call
the  \emph{generalized $h$-th Hilbert point associated to the couple  $(X, V)$.}

\begin{definition}\label{genhilb}
With the above notation, we say that the  couple $(X, V)$ is {\em Hilbert (semi)stable} if its generalized $h$-th Hilbert points are GIT (semi)stable for infinite $h\in \mathbb N$.
\end{definition}

\begin{remark}\label{hilbfinito}
Let   $(X, V)$ be as above.
Consider the factorization of the induced map through the image
$$
\xymatrix{X\ar@{-->}[r]& \overline X \,\,\ar@{^{(}->}[r]^{j}&\pr^r.}
$$
Set $\overline \cD=j^*(\cO_{\pr^r}(1))$ and let $\overline V\subseteq H^0(\overline X,\overline \cD)$
 be the linear systems associated to $j$.
The homomorphism (\ref{omo}) factors as follows:
$$\sym^hV\cong\sym^h\overline V\stackrel{\overline \varphi_h}{\longrightarrow} H^0(\overline X,\overline{\cD}^{\otimes h})
\hookrightarrow H^0(X, \cD^{\otimes h}),$$
where the homomorphism $\overline \varphi_h$ is the $h$-th Hilbert point of the
embedding $j$;
notice that, by Serre's vanishing theorem, this homomorphism is onto
(and, in particular, $G_h=H^0(\overline X,\overline{\cD}^{\otimes h})$)  for large enough $h$.
The generalized $h$-th Hilbert point of $(X, V)$ is therefore naturally identified
with the $h$-th Hilbert point of $(\overline X,\overline V)$, and  the generalized Hilbert stability
of $(X,V)$ coincides with the classical Hilbert stability
of the embedding $j$.
\end{remark}

Now consider a fibred variety $f\colon X \longrightarrow Y$, where the base $Y$ is smooth but not necessarily of dimension $1$.
Let ${\cal L}$ be a line bundle on $X$, and let ${\cal G}\subseteq f_*{\cal L}$ be a subsheaf of rank $r$.
Consider the homomorphism of sheaves
$\sym^h\cG\longrightarrow f_*\cL^{\otimes h}$ and, as usual, call $\cG_h$ its image.

\begin{theorem}[Cornalba-Harris]\label{CoHa}
With the above notation, suppose that for general $y\in Y$ the $h$-th generalized Hilbert point of the fibre ${G}_y:={\cal G}\otimes \kapp(y)\subseteq H^0(F, {\cal L}_{|F})$ is semistable.

Then the line bundle
$$
\mathcal L_h:=\det(\cG_h)^{\otimes r}\otimes (\det \cG)^{-\otimes h\rk\cG_h}
$$
is pseudo-effective.
\end{theorem}
The above result  is the key point of the proof of \cite[Theorem 1.1]{CH}.
%
%
%
%
%
In particular, when the base $Y$ is a smooth curve, we obtain the following inequality
\begin{equation}\label{importante}
\rk \cG \deg \cG_h-h \deg \cG\rk \cG_h\geq 0,
\end{equation}

In the general case with base of arbitrary dimension it is possible, under some assumptions, to compute the first Chern class of $\cL_h$ as a polynomial in $h$ with coefficients in $CH_1(Y)_\Q$ and to conclude that its leading term is a pseudoeffective class (\cite[Theorem  1.1]{CH} and \cite[Corollary 1.6]{LS}).

 \smallskip

Applying the results of Section \ref{risultati}, we obtain the following condition for $f$-positivity, which provides an improvement of Theorem 1.1 of \cite{CH} in the case of  $1$-dimensional base.

\begin{theorem}\label{corCH}
With the notation above, suppose that the base $Y=B$ is a curve.
Suppose that the sheaf $\cG$ is either generating, or it satisfies the conditions of Proposition \ref{conto1,5} or  \ref{conto2}.
Suppose moreover that for general $t\in B$ the fibre ${G}_t\subseteq H^0(F, {\cal L}_{|F})$ is Hilbert semistable.
Then $(\cL,\cG)$ is $f$-positive.
\end{theorem}
\begin{proof}
Apply Theorem \ref{CoHa} above, and Theorem \ref{conto} and Proposition \ref{conto1,5} and \ref{conto2}.
\end{proof}


\subsubsection*{Bost's result: Chow stability}

We now describe a result of Bost, which is almost equivalent to the one of Cornalba-Harris, except that it uses as assumption the Chow stability on the general fibres.
Moreover it has to be mentioned that Bost's result holds in positive characteristic.

Let us first recall some definitions.
Let $X$ be an $n$-dimensional variety together with a finite morphism of degree $a$ in the projective space $\varphi\colon X\longrightarrow \pr^r$ associated to a linear system $V\subseteq H^0(X, \cD)$.
Consider
$$Z(X):=\{n-\mbox{spaces } \pi \mbox{ of  } V \mid Ann(\pi)\cap \varphi(X) \not = \emptyset \}\subset  Gr(n,V).$$
The set $Z(X)$ is an hypersurface  of degree $d=\deg \varphi /a$, in the grassmanian $Gr(n,V)$.
The homogeneous polynomial $F_X\in H^0( Gr(n,V), \cO_{Gr(n,V)}(d))$ representing $Z(X)$ is the {\em Chow form} of $(X,V)$ and the  {\em Chow point} of $(X,V)$ is the class of $F_X$ in $\pr (H^0( Gr(n,V), \cO_{Gr(n,V)}(d)))$

The couple $(X, V)$ is  {\em Chow (semi)stable} if  its Chow point is GIT (semi)stable with respect to the natural $SL(V)$ action.

\begin{remark}\label{Chow-image}
Note that $X$ is Chow (semi)stable if and only if the cycle $mX$ is, for any integer $m$: see for instance \cite{bost}, proof of Proposition 4.2.
So, in particular, the Chow (semi)stability of $(X,V)$ as above coincides with the Chow (semi)stability of the cycle image $\varphi_*(X)$ together with the linear system of the immersion induced by $\varphi$. This fact should be compared with the behavior of the Hilbert stability described in Remark \ref{hilbfinito}
\end{remark}

\begin{remark}\label{Chow-Hilbert}
In \cite[Corollary 3.5]{Mor}, it is proven that Chow stability implies Hilbert stability, while for {\em semi}stability, the arrows are reversed. Although it is not known an example in which this two stabilities do not correspond, nothing is known about the converse implication, except for asymptotic results (see Remark \ref{risultati-hilbert-stability} below).
\end{remark}
Hence, we can apply Theorem \ref{corCH} if we replace the assumption of Hilbert semistability with Chow stability, but we can not assume Chow semistability.

In \cite{bost},  Bost has proven an arithmetic analogue to  the theorem of Cornalba and Harris, assuming the Chow semistability of the maps on the general fibres. The  geometric counterpart of Bost's result in the case when the base is $1$-dimensional is the following.
Consider as usual a fibred variety $f\colon X \longrightarrow B$.
Let ${\cal L}$ be a line bundle on $X$, and let ${\cal G}\subseteq f_*{\cal L}$ be a subsheaf of rank $r$.

\begin{theorem}[\cite{bost} Theorem 3.3]\label{teobost}
With the above notation, suppose that
\begin{enumerate}
\item for $t\in B$ general, the fibre $G_t:=\mathcal G\otimes\kapp(t)\subseteq H^0(F, \cL_{|F})$ is base-point free;
\item if $\alpha\colon F\longrightarrow \pr^r$ is the morphism induced, the cycle $\alpha_*(F)\in Z_p (\pr^r)$ is Chow semi-stable;
\item the line bundle $\cL$ is relatively nef.
\end{enumerate}
Then the couple
$(\cL,\mathcal G)$ is $f$-positive.
\end{theorem}




\begin{example}\label{slopeCH}
Let us prove the slope inequality for fibred surfaces via these methods.
Let $f\colon S\longrightarrow B$ be a relatively minimal fibred surface of genus $g\geq 2$.
Recall that the relative dualizing sheaf $\omega_f$ is nef \cite{Bv}.
The slope inequality for relatively minimal fibred surfaces now follows right away from Proposition \ref{conto1,5}, using the fact that the restriction of $\omega_f$ to the general smooth fibre is Hilbert  and Chow semistable (Remark \ref{risultati-hilbert-stability} above), and base-point free.
An alternative proof can be obtained using Proposition \ref{conto2}, by proving, as in \cite{CH}, that condition $(\star)$ holds.

Let us now refine the computation in the case of a relatively minimal {\em nodal} fibred surface.
In this case we have given in Example \ref{moving-surfaces} an explicit description of the moving and the fixed part of
$(\omega_f, f_*\omega_f)$. Recall that the moving part is $\cM\cong\nu^*\omega_f(-D)\otimes\cO_{\widetilde S}(-E)$, where $D$ is the union of all socket type components, $\nu\colon \widetilde S \longrightarrow S$ is the blow up of $S$ in the disconnecting nodes  of the fibres of $f$ that do not belong to a socket type component and $E$ is the exceptional divisor of $\nu$.
Let $\widetilde f=\nu\circ f$ be the induced fibration.
From the proof of Proposition \ref{conto1,5}, we can derive the following inequality:
$$
0\leq M^2-2\mu(\widetilde f_*\omega_{\widetilde f})\deg \omega_{|\widetilde F}=  K_f^2+D^2+E^2-2K_f D-4\frac{(g-1)}{g}\deg f_*\omega_f.
$$
Let us compute explicitly the term $D^2+E^2-2K_fD$. Let $n$ be the total number of disconnecting nodes contained in the fibres, $k$ the number of nodes lying on a socket type component and $l=n-k=-E^2$.
Let $r$ be the number of connected components of socket type in the fibres, so that $D=D_1+\ldots+D_r$ with the $D_i$'s connected and disjoint.
Then we have that $K_fD=-2r+k$, so that $D^2+E^2-2K_fD= 3k-4r+l$.
Note that the condition of relative minimality is equivalent to $2r\leq k$, so we obtain inequality
\begin{equation}\label{refslope1}
K_f^2\geq 4\frac{g-1}{g}\chi_f+n.\footnote{For the reader familiar with the moduli space of curves $\overline\cM_g$, this inequality means that the divisor $g\kappa_1-4(g-1)\lambda-g\sum_{i>0}\delta_i\sim (8g+4)\lambda - g\delta_0-2g\sum_{i>0}\delta_i$ is nef outside the boundary  $\partial \overline\cM_g$}
\end{equation}
In particular any fibred surface satisfying the slope {\em equality} necessarily has all fibres free from disconnecting nodes.
It is interesting to compare this result with the inequalities obtained via Xiao's method (Example \ref{slopeX}) and with Moriwaki's method (Example \ref{slopeM}).
\end{example}

\subsection{Some remarks on GIT stabilities and applications}\label{risultati-hilbert-stability}

It comes out the interest in understanding when a variety, endowed with a map in a projective space, is Hilbert or Chow semistable. The following is a (without any doubt non-complete) list of cases where Hilbert (or Chow) semistability is known. In this list any time we use the term ``stability'' without specification, we mean that both the Hilbert and the Chow (semi)stabilities are known to coincide.
\begin{itemize}
\item Homogeneous spaces embedded by complete linear systems are semistable; abelian varieties embedded by complete linear systems are semistable \cite{kempf}.
\item  Linear systems on curves: if $C$ is a smooth curve of genus $g\geq 2$, the canonical embedding is Chow semistable, and it is Chow stable as soon as $C$ is non-hyperelliptic. Any line bundle of degree $d\geq 2g+1$ induces a Chow stable embedding \cite{Mum}.
Deligne-Mumford stable curves are semistable for the linear system induced by the $m$-th power of the dualizing sheaf for $m\geq 5$ (\cite{Mum}, \cite{Gie}, \cite[Chap.4, Sec.C]{H-M}). See \cite{schubert} \cite{Hassett-Hyeon}  for curves Chow stable with respect to lower powers of the dualizing sheaf.
\item Morrison in \cite{Mor} studies the Chow stability of ruled surfaces in connection with the $\mu$-stability of the associated rank $2$ vector bundle: he proves that  if $\cE$ is a stable rank 2 bundle on a smooth curve $C$ then the ruled surface $\pi\colon \pr(\cE)\longrightarrow C$  is Chow stable with respect to the polarisation $\cO_{\pr(\cE)}(1) \otimes \pi_*\cO_C(k)$ for $k\gg 0$.  Seyyedali in \cite{machenomeciha} extends the results of Morrison to higher rank vector bundles and to higher dimensional bases. See also \cite{KellerRoss} for another generalization.
\item General  $K3$ surfaces:  a $K3$ surface with Picard number $1$ and degree at least $12$ is Hilbert semistable \cite{morrison-K3}.
\item Hypersurfaces:  in \cite[Prop. 4.2]{GIT} it is proven that  smooth hypersurfaces of $\pr^n$ of degree $\geq 3$ are stable. In \cite{Mum} it is studied the stability of (singular) plane curves and surfaces in $\pr^3$. A hypersurface $F \subset \pr ^r$ of degree $d \geq r+2$ and only log terminal singularities is Hilbert semistable \cite{Tian}.
\item Higher codimensional varieties: Lee   \cite{Lee} proved that a subvariety $F \subset \pr ^r$ of degree $d$ is Chow semistable as far as  the log canonical threshold  of its Chow form is greater or equal to $ \frac{r+1}{d} $ (resp. $>$ for stability).
In \cite{BMV} both the Chow and the Hilbert stability of curves of degree $d$ and arithmetic genus $g$ in $\pr^{d-g}$  are studied.
\end{itemize}
A lot of remarkable results --due to Gieseker, Viehweg and many others-- are known regarding {\em asymptotic} stability: given a line bundle $\cD$ and a linear subsystem $V\subseteq H^0(X,\cD)$, this is the stability of the couple $(\cD^{\otimes h}, V_h)$, for high enough $h$, where
$$V_h:= \mbox{Im} (\sym^hV\longrightarrow H^0(X, \cD^{\otimes h})).$$
In this case Hilbert and Chow stability  have been proved to be equivalent by Fogarty \cite{fog} and Mabuchi \cite{mabuchi}.
There are beautiful results due to Donaldson, Ross, Thomas and many others relating asymptotic Chow stability to differential geometry properties, such that the existence of a constant scalar curvature metric.
Unfortunately, if a bound is not known on the power of the line bundle, the Cornalba-Harris theorem does not give interesting consequences: if a couple $(\cG,\cL)$ is asymptotically semistable on a general fibre, then the Cornalba-Harris theorem implies that $L^n\geq 0$.

\medskip

On the other hand, it has come out recently, also in relation with the minimal model program for the moduli space of curves initiated in \cite{Hassett-Hyeon}, the interest in the stability of the $h$-th Hilbert point for {\em fixed $h$}. The main result obtained in this topic is that general canonical and bicanonical curves have  the $h$-th Hilbert point semistable for $h\geq 2$ \cite{AFS}.

The Cornalba-Harris method can be applied with this kind of assumption.
For instance we can prove the following result (cf \cite{FJ} for $h=2$).

\begin{proposition}\label{hilb-finito-prop}
Let $f\colon S\longrightarrow B$ be a relatively  minimal non-hyperelliptic fibred surface of genus $g\geq 2$.
Suppose that the $h$-th Hilbert point of a general fibre $F$ with its canonical sheaf is semistable (with $h\geq 2$).
Then the following inequality holds
\begin{equation}\label{h-finito}
K_f^2\geq 2\frac{2(g-1)h^2+(1-g)h-g}{gh(h-1)}\chi_f.
\end{equation}
\end{proposition}
\begin{proof}
With the usual notation, we choose $\cL=\omega_f$ and $\cG=f_*\omega_f$.
Then by the assumption, using Theorem \ref{CoHa}, we have that
$ \rk \cG \deg \cG_h-h \deg \cG\rk \cG_h\geq 0$.
By Riemann-Roch, $\rk \cG_h=(2h-1)(g-1)$, and $\deg \cG_h=\frac{h(h-1)}{2} K_f^2+\chi_f$, and the computation is immediate.
\end{proof}
The computations with higher powers of the relative canonical sheaf gives worse inequalities than the slope one.

\begin{remark}
By a result of Fedorchuck and Jensen \cite{FJ} (that improves the result in \cite{AFS}), the best inequality in equation (\ref{h-finito}), reached for $h=2$, holds for relatively minimal fibred surfaces whose general fibres are non-hyperelliptic curves of genus $g$ whose canonical image does not lie on a quadric of rank 3 or less. In particular this is the case for fibred surfaces of even genus whose general fibres are trigonal with Maroni invariant $0 $ (ibidem. and \cite{BS2}).
It is quite interesting to notice that this very same bound is obtained by Konno in \cite[Lemma 2.5]{konnohyp}  under the assumption that the the pushforward sheaf $f_*\omega_f$ is $\mu$-semistable.
\end{remark}

From the above proposition we can derive a new proof of the following result (cf. \cite[ Theorem. 4.12]{CH} and \cite[Prop. 2.4]{LS}). The same result follows from the computation contained in Remark \ref{terminedopo}.
\begin{corollary}
If a relatively minimal non-locally trivial fibred surface of genus $g\geq 2$ reaches the slope inequality, then it is hyperelliptic.
\end{corollary}
\begin{proof}
Observe that the function of $h$ appearing in inequality (\ref{h-finito}) is strictly decreasing and --of course-- it tends to the ratio of the slope inequality $4(g-1)/g$ for $h\mapsto \infty$.
So, for any non-hyperelliptic fibration in the conditions of the theorem, a strictly stronger bound than the slope one is satisfied.
\end{proof}

\subsubsection*{A new stratification of $\cM_g$}\label{hilbertstratification}
It is widely believed (see for instance \cite{BS2}, \cite{konnocliff}) that there should exist a lower bound for the slope of fibred surfaces increasing with the gonality of the general fibres (under some genericity assumption). This conjecture, however, is only proved for some step: hyperelliptic fibrations (the slope inequality), trigonal fibrations  \cite{BS2}, \cite{AnAn} and fibrations with general gonality \cite{konnocliff}, \cite{HE}. Recently Beorchia and Zucconi \cite{Beo-Zucc} have proved some results also on fourgonal fibred surfaces.

Let us consider the following open subsets of $\cM_g$
$$
\cS_h:=\left\{ [C]\in \cM_g \mbox{ such that the $k$-th Hilbert point is semistable for }k\geq h\right\}.
$$
Clearly $\cS_i\subseteq \cS_j$ for $i\leq j$, and for some $m\in \mathbb N$ the sequence becomes stationary, i.e. $\cS_i= \cS_j$ for every $i, j\geq m$ (cf. \cite{Gie}).
If we consider the subsets $\cS_2, \cS_3\setminus \cS_2, \ldots , \cS_m\setminus \cS_{m-1}$, it seems possible that these provide an alternative stratification of $\cM_g$ minus the hyperelliptic locus.
For such a stratification, a lower bound for the slope increasing with the dimension of the strata would be provided by Proposition \ref{hilb-finito-prop}.
However, it does not seem clear, at least to the authors, to give a geometrical characterization of the curves lying in $\cS_i\setminus \cS_{i-1}$, and an estimate on the codimensions of these strata.


\subsection{Xiao's method: the Harder-Narashiman filtration}\label{XIAO}

As we have seen in the previous section, $\mu$-semistability of $\cG$ implies $f$-positivity.
What about the case when the sheaf $\cG$ is not semistable as a vector bundle?
We describe here a method  based on Miyaoka's Thorem \ref{teomiyaoka}, which exploits the Harder-Narashimann filtration of the sheaf $\cG$.

The main idea is given by Xiao in \cite{X}, where he uses the method in the case of fibred surfaces.
Later on,  Ohno \cite{ohno} and Konno \cite{konnotrig} extended the method to higher dimensional fibred varieties over curves.
We present here a compact version of the general formula (see Proposition \ref{xiaoinequality} below).

\medskip

We need to recall the definition of the {\em  Harder-Narashimann filtration} of a vector bundle $\cG$ over a curve $B$:
it is the unique filtration of subbundles
$$0=\cG_0\subset \cG_1\subset \ldots \subset \cG_l=\cG$$
satisfying the following assumptions
\begin{itemize}
\item for any $i=0, \ldots l$ the sheaf $\cG_i/\cG_{i-1}$ is $\mu$-semistable;
\item if we set $\mu_i:= \mu(\cG_i/\cG_{i-1})$, we have that $\mu_i>\mu_{i-1}$.
\end{itemize}
Note that $\mu_1>\mu(\cE)>\mu_l$, unless $\cG$ is $\mu$-semistable, in which case $1=l$ and these numbers are equal. If $H$ is a divisor associated to the tautological line bundle of $\mathbb{P}(\cG)$ and $\Sigma$ is a general fibre then an $\mathbb{R}$-line bundle $H-x\Sigma$ is pseudoeffective if and only if $x \leq \mu_1$ \cite[Cor. 3.7]{Nakayama} and it is nef if and only if $x \leq \mu_l$ \cite{miyaoka}.

\medskip

As usual, consider an $n$-dimensional fibred variety  $f\colon X \longrightarrow B$ and be a line bundle  $\cL$ on $X$.
Let $F$ be a general smooth fibre of $f$.
Consider ${\mathcal G}\subseteq f_*\cL$ a subbundle, and its corresponding Harder-Narashiman filtration as above.
Set $r_i=\rm{rank}\cG_i$.

For each $i=1,...,l$, we consider the pair $(\cL,{\cG}_i)$ as in \ref{firstdefinitions} and a common resolution of indeterminacies $\nu : {\widetilde X} \longrightarrow X$.
Let $M_i$ be the moving part of $(\cL, {\cG}_i)$, and let $N_i=M_i-\mu _i F$.
By Miyaoka's theorem \ref{teomiyaoka} we have that $N_i$ is a nef $\mathbb{Q}-$divisor (not necessarily effective).
The linear system $P_i:={N_i}_{|\widetilde F}$
is free from base points and induces a map $\phi _i: \widetilde F \longrightarrow {\mathbb P}^{r_i-1}$. By construction we have $P_l \geq P_{l-1} \geq ... \geq P_2 \geq P_1$.
Define $a_{l+1}=0$ and $N_{l+1}=N_l$. Then, we can state the generalized Xiao's inequality as follows. We refer to \cite{konnotrig} for proofs.

For any set of indexes $I=\{i_1,...,i_m\} \subseteq \{1,2,....,l\}$, define $i_{m+1}=l+1$ and consider the partition of $I$ given by
$$
I_s=\{i_k \, | \, k=1,...,m  \, \, {\rm such} \ {\rm that} \,\, {\rm dim}\phi _{i_k}({\widehat F})=s \, \}.
$$
Define now $b_{n}=l+1$ and decreasingly

$$b_s=\left\{
           \begin{array}{cc}
             {\rm min} I_s & {\rm if} \,\, I_s \neq \emptyset \\
             b_{s+1} & {\rm otherwise}. \\
           \end{array}\right.$$

\begin{proposition}[Xiao, Konno]\label{xiaoinequality}
With the above notation, assume the $\cL$ and $\cG$ are nef. Then the following inequality holds
\begin{equation}\label{XiaoInequality}
L^n=(\nu ^* L)^n \geq N_{l+1}^n\geq\left( \sum _{s=n-1}^{1} (\prod_{n-1 \geq k>s}P_{b_k}) \sum_{j \in I_s}(\sum_{r=0}^{s}P_{j}^{s-r} P_{j+1}^{r})\right)(\mu_{j}-\mu_{j+1}).
\end{equation}
\end{proposition}

\begin{remark}
As we see Xiao's method does not give as a result $f$-positivity, but an inequality for the top self-intersection $L^n$ that has to be interpreted case by case. On the other hand, it basically only has one hypothesis: the nefness of $\cL$ and of $\cG$. However, as we will see in \ref{LS&X} we can derive results on $\cG=f_*\cL$ even if it is not a nef vector bundle.
One of the contributions of this article is to frame Xiao's result in a more general setting, and to prove that with the right stability condition in the couple  $(F, \cG_{|F}) $, for $F$ general, Xiao's method produces $f$-positivity, at least in the case of dimension $2$.
\end{remark}

\begin{example}\label{slopeX}
Let us describe how inequality (\ref{XiaoInequality}) implies the slope inequality in the case of fibred surfaces.
We use  the above formula for $n=2$, $\cL=\omega _f$, $\cG=f_*\omega_f$ and the sets of indexes $I=\{1,...,l\}$ and $I'=\{1,l\}$. If we call $d_i={\rm deg}P_i$ inequality (\ref{XiaoInequality}) becomes, respectively
$$
K_f^2 \geq \sum_{i=1}^{l}(d_i+d_{i+1})(\mu _i - \mu_{i+1}),
$$
$$
K_f^2 \geq (d_1+d_l)(\mu _1 - \mu _l)+2d_l \mu_l\geq d_l(\mu_1+\mu_l)=(2g-2)(\mu_1+\mu_l).
$$
Let us note that by Cifford's theorem we have inequality $d_i \geq 2r_i-2$.
Observing now that $r_{i+1} \geq r_i+1$, and that ${\rm deg} f_*\omega_f=\sum_{i=1}^{l}r_i(\mu_ i-\mu _{i+1})$, we  obtain straight away the slope inequality
$$
K_f^2 \geq 4\frac{g-1}{g}{\rm deg}f_*\omega_f.
$$
In fact, the above proof gives an inequality  for $N_l^2$. In the case of nodal fibrations, using the same notations as in Example \ref{slopeCH}, since $N_l=\nu^*(K_f(-D))(-E)$, we obtain the inequality $N_l^2\leq K_f^2-n$, which gives the very same inequality (\ref{refslope1}) obtained via the Cornalba-Harris method.
\end{example}

\begin{example}\label{xiao3folds}
It could be interesting to have explicitly written the case $n=3$ for the complete set of indexes $\{1,...,l\}$. Assume that $N_l$ induces a generically finite map on the surface $F$. Hence we have $I_{2}\neq \emptyset$ and so
$$L^3 \geq 3P_l^2\mu_l+(P_l^2+P_l P_{l-1}+P_{l-1}^2)(\mu_{l-1}-\mu_l)+...$$$$...+(P_{b_2+1}^2+P_{b_2+1}P_{b_2}+P_{b_2}^2)(\mu_{b_2}-\mu_{b_2+1})+
$$$$+P_{b_2} [(P_{b_2}+P_{b_2-1})(\mu_{b_2-1}-\mu_{b_2})+.....+(P_{b_1+1}+P_{b_1})(\mu_{b_1+1}-\mu_{b_1})].$$

\noindent Observe that $b_1=1$ except for the case $r_1=1$ where $b_1=2$.

Since the linear systems induced by $P_i$ for $i=b_1,\ldots ,b_2-1$ map $F$ onto curves $C_i$, we have a chain of projections between these curves in such a way that the fibration part of the Stein factorization of the maps $F \longrightarrow C_i$ are the same. Hence we have a fibration
$$\pi: F \longrightarrow C.$$
\noindent Call $D$ the general fibre, and let $Q_i$ be base point free linear systems on $C$ such that $P_i=\pi^*Q_i$ of rank $h^0(C, Q_i)\geq r_i=\rk \cG_i$ and degrees which we call $d_i$. Writing this information and using that for all $j$
$$P_{j+1}^k P_j^l \geq  P_{j+1}^{k-1} P_j^{l+1},$$ since $P_j \leq P_{j+1}$ and they are nef, we obtain a simplified (and weaker) version of the previous inequality:
\begin{equation}\label{xiao3foldssimplificada}
L^3 \geq 3P_l^2\mu_l+P_{l-1}^2(\mu_{l-1}-\mu_{l})+...+P_{b_2}^2(\mu_{b_2}-\mu_{b_2+1}))+$$$$+2\lambda(d_{b_2-1}(\mu_{b_2-1}-\mu_{b_2})+...+d_{b_1}(\mu_{b_1}-\mu_{b_1+1})),
\end{equation}
\noindent where $\lambda=DP_{b_2}$.

\end{example}


\subsection{Moriwaki's method: $\mu$-stability on the fibres}\label{MORIWAKI}

In this paragraph we shall restrict ourselves to the case $n=2$; see Remark \ref{bogomolov-high-dim} below for a discussion on higher-dimensional results.
Let $X=S$ be a smooth surface. We need the following fundamental result due to Bogomolov, which can be found in \cite{B2}.

\begin{definition} Let ${\cal E}$ be a torsion free sheaf over $S$. The class
$$\Delta ({\cal E}):= 2\,\rk {\cal E} c_2({\cal E})-(\rk {\cal E}-1)c_1^2({\cal E})\in A^2_{\Q}(S)$$
is the {\em discriminant} of the vector bundle $\cE$. Let $\delta(\cE)$ denote its degree.
\end{definition}

\begin{theorem}[Bogomolov Instability Theorem]\label{BIT}
With the above notation, if $\delta({\cal E})<0$ then there exists a saturated subsheaf ${\cal F}\subseteq {\cal E}$ such that the class
$$D=\rk {\cal F}\, c_1({\cal E})-\rk {\cal E}c_1({\cal F})$$ belongs to the positive cone $K^+(S)$ of $\mbox{\upshape Pic}_\Q(S)$.
\end{theorem}
\noindent Recall that the positive cone $K^+$ is defined as follows: consider the (double) cone
$$K(S)=\{A\in N^1(S)_\Q\mid A^2>0\}\subset N^1(S)_\Q.$$
The cone $K^+(S)$ is the connected component of $K(S)$ containing the ample cone.

\medskip

\begin{remark}\label{H-stabile}
Recall the definition of semistable sheaf in higher dimension: if $X$ is a variety of dimension $n$ and $\cF$ a locally free sheaf on $X$, let $\cH$ be an ample line bundle on $X$. We say that $\cF$ is $\cH$-(semi)stable if for any proper subsheaf $0\not =\cR\subset \cF$
$$
\frac{c_1(\cR)\cdot H^{n-1}}{\rk \cR}\leq \frac{c_1(\cF)\cdot H^{n-1}}{\rk \cF} \quad\quad  (\mbox{resp. } <),
$$
where $H$ is the class of $\cH$.
In particular from the strong instability condition provided by the theorem above, we have that if $\cE$ is $\cH$-semistable with respect to {\em any} ample line bundle $\cH$ on $S$, then $\delta(\cE)\geq 0$.
\end{remark}


The argument of Moriwaki relies on two key observations. The first is the following: if the surface $S$ carries a fibration, then, in order to ensure the non-negativity of $\delta(\cE)$ for a vector bundle $\cE$, one can assume that $\cE$ is semistable {\em on the general fibres} of $f$.

\begin{proposition}[\cite{Mor1} Theorem 2.2.1]\label{corollary bogomolov}
Let us consider a fibred surface $f\colon S\longrightarrow B$.
Let $\cE$ be a sheaf on $S$ such  that the restriction of $\cE$ on a general fibre of $f$ is a $\mu$-semistable sheaf.
Then $\delta(\cE)\geq 0$.
\end{proposition}
\begin{proof}
Suppose by contradiction that $\delta ({\cal E})<0$. Then by the Bogomolov Instability Theorem there exists a saturated subsheaf ${\cal F}\subseteq {\cal E}$ such that the divisor
$D=\rk {\cal F}\, c_1({\cal E})-\rk {\cal E}c_1({\cal F})$ satisfies that $D^2>0$.
As a fibre $F$ is nef, and $F^2=0$, by the Hodge Index Theorem \cite[sec.IV, Cor. 2.16]{BHPVdV}, we have that
 $$
 0< D\cdot F=\rk {\cal E} \deg {\cal F}_{|F}-\rk {\cal F}\deg {\cal E}_{|F}.
 $$
 So ${\cal F}_{|F}$ is a destabilizing subsheaf of ${\cal E}_{|F}$, against the assumption.
\end{proof}

\begin{remark}\label{paragonestabil}
It is worth noticing that the hypotesis of the above proposition that the restriction $\cE_{|F}$ is $\mu$-semistable on the general fibres $F$ does not imply, neither is implied, by some semistability of the sheaf $\cE$ on $S$. Indeed, we can say only the following:
\begin{itemize}
\item Let $\cH$ be an ample line bundle on $S$ and $C$ be a general  curve in $|\cH^{\otimes d}| $ with $d\geq 1$. Suppose that $\cE_{|C}$ is $\mu$-semistable, then $\cE$ is $\cH$-semistable. Indeed if $\cF$ would be an $\cH$-destabilizing subsheaf of $\cE$, then $\cF_{|C}$ would be destabilizing for $\cE_{|C}$, because
$$\mu (\cF_{|C})=\frac{\deg \cF_{|C}}{\rk\cF_{|C}}=d \frac{\deg(c_1(\cF)\cdot H)}{\rk\cF}>d\frac{\deg(c_1(\cE)\cdot H)}{\rk\cE}=\frac{\deg \cE_{|C}}{\rk\cE_{|C}}=\mu(\cE_{|C})$$
\item If $\cE$ is $\cH$-semistable with respect to some ample line bundle $\cH$, and $C$ is a general curve in $|\cH^{\otimes m}|$, for sufficiently large $m$, then $\cE_{|C}$ is $\mu$-semistable \cite{MR}.
\end{itemize}
Note that as a fibre $F$ of any fibration $f\colon S\longrightarrow B$ satisfies $F^2=0$, it cannot be ample.
However, if the fibration is rational (i.e. $B\cong \pr^1$), the conditions above can hold true after some blow down of sections of the fibration.
\end{remark}

Let us consider now a fibred surface $f\colon S\longrightarrow B$, a line bundle $\cL$ and a rank $r$ subsheaf $\cG\subseteq f_*\cL$.
The second point of Moriwaki's argument, using our terminology, relates $\delta(\cE)$ to $e(\cL,\cG)$, for a suitably chosen vector bundle $\cE$, as follows.

\smallskip

Let $\cM$ be the kernel of the evaluation morphism $f^*\cG\subset f^*f_*\cL\longrightarrow\cL $.
The following  is a generalization of a computation contained in \cite{Mor1}.

\begin{proposition}\label{e=delta}
With the above notation, if either
\begin{itemize}
\item[(a)] the sheaf $\cG$ is generating in codimension 2, or
\item[(b)] the line bundle $\cL$ is $f$-nef, and $\cG$ has base locus vertical with respect to $f$,
\end{itemize}
then $$\delta(\cM)\leq e(\cL, \cG).$$
\end{proposition}
\begin{proof}
Let us call $\cK$ the image of the evaluation morphism, so that we have the following exact sequence
 $$
 0\longrightarrow {\cal M}\longrightarrow f^*{\cal G}\stackrel{\varphi}{\longrightarrow}\cK \longrightarrow 0.
 $$
Note that, with the notations of Section \ref{firstdefinitions},
$c_1(\cK)=c_1(\cL(-D))$, where $D$ is the fixed locus of $\cG$, and $c:=c_2(\cK)=-E^2\geq 0$, where $E$ is as in \ref{firstdefinitions}.
Indeed, $c_2(\cK)$ is the length of the isolated base points of the variable part (with natural scheme structure) \cite{A-K}.
So we have
\begin{itemize}
\item $c_1(\cM)=f^*c_1(\cG)-c_1(\cL(-D))$;
\item $c_2(\cM)=-c_1(\cM)c_1(\cL(-D))-c$.
\end{itemize}
Hence $\deg c_2(\cM)= (L-D)^2-\deg \cG(L-D)F-c$. Let $r=\rk \cG$, so that $\rk \cM=r-1$.
We can make the following computation
$$
\begin{array}{ll}
\displaystyle  \delta(\cM)& =2(r-1) \left[(L-D)^2-\deg\cG (L-D)F -c\right] +\smallskip \\
\displaystyle& -(r-2)\left[(L-D)^2-2\deg \cG (L-D)F\right]=\smallskip \\
\displaystyle& = r(L-D)^2-2\deg \cG(L-D)F-2(r-1) c.\smallskip\\
\end{array}
$$
In case (a) $D=0$ and we thus obtain $\delta(\cM)=e(\cL, \cG) -2(r-1)c\leq e(\cL, \cG)$.
In case assumption  (b) holds, observe that $(L-D)^2=L^2-2LD+D^2\leq L^2$; indeed being $D$ effective and vertical, we have  $D^2\leq 0$  by Zariski's Lemma, and  $LD\geq 0$ because $L$ is supposed to be $f$-nef.
On the other hand, $(L-D)F=LF$ again because $D$ is vertical.
Hence, we still obtain the desired inequality, and the proof is concluded.
\end{proof}

Combining Proposition \ref{corollary bogomolov} and Proposition \ref{e=delta} we get immediately the following result

\begin{theorem}\label{bogomolov stability}
With the notation above, suppose that the restriction $\cM_{|F}$ to a general  fibre $F$ is a semistable sheaf on it.
Suppose moreover that one of the following conditions holds.
\begin{itemize}
\item the sheaf $\cG$ is generating in codimension 2;
\item the line bundle $\cL$ is $f$-nef, and $\cG$ has base locus vertical with respect to $f$.
\end{itemize}
Then the couple $(\cL, {\cal G})$ is $f$-positive.
\end{theorem}

\begin{example}\label{slopeM}
Let $f\colon S\longrightarrow B$ be a  relatively minimal non-locally trivial  fibred surface. Let us prove the slope inequality via Moriwaki's method.
Let us consider the couple $(\omega_f, f_*\omega_f)$, and let $\cM$ be the kernel sheaf of the evaluation morphism $f^*f_*\omega_f\longrightarrow \omega_f$.
The hypoteses of the above theorem are satisfied.
Indeed, the assumption that $\cM_{|F}$ is a semistable sheaf has been proved by Pranjape and Ramanan in \cite{PR}.
The evaluation morphism $f^*f_*\omega_f\longrightarrow \omega_f$ can fail to be surjective on some vertical divisor, so that Theorem \ref{bogomolov stability} can be applied with assumption (b) holding, and leads to the slope inequality (see also \cite{A-K}).

In case $f$ is a nodal fibration, we can obtain a finer inequality, as follows.
From the computations of Proposition \ref{e=delta}, using the results contained in Example \ref{moving-surfaces}, we have that
$$0\leq \delta(\cE)= gK_f^2-4(g-1)\chi_f - g(3k-4r)-2(g-1)l,$$
where, as in Example \ref{slopeCH}, $k$ is the number of disconnecting nodes lying on components of socket type of the fibres, $l$ is the number of the others disconnecting nodes, and $r$ is the number of components of socket type. So, we get
$$
K_f^2\geq 4\frac{g-1}{g}\chi_f + k +2\frac{g-1}{g}l.
$$
Note that this inequality is slightly better than the one obtained in Examples \ref{slopeCH} and \ref{slopeX} using respectively the Cornalba-Harris and the Xiao methods.
\end{example}

\begin{remark}
Moriwaki  in \cite{Mor1} uses, for nodal fibred surfaces $S\longrightarrow B$, as line bundle ${\cal L}$ on $S$ an ad hoc modification of the relative canonical bundle $\omega_f$ on the singular fibres, and  as ${\cal G}$ the whole $f_*\cL$.
The result is an inequality, involving some contributions due to the singular fibres, stronger than the one  obtained in Example \ref{slopeM}.
\end{remark}

From  Proposition \ref{e=delta}, combining it with Remark \ref{H-stabile}, we can straightforwardly deduce the following condition for $f$-positivity.

\begin{theorem}\label{f-positivity-high-stability}
With the notations above, if the kernel of the evaluation morphism
$$f^*\cG\longrightarrow \cL$$
 is $\cH$-semistable with respect to an ample line bundle $\cH$ on $S$, then $(\cL, \cG)$ is $f$-positive.
\end{theorem}
This result is not implied by Moriwaki's Theorem \ref{bogomolov stability}, by what observed in Remark \ref{paragonestabil}.

\begin{remark}\label{bogomolov-high-dim}
It would be nice to be able to extend Moriwaki's method to higher dimensions.
Thanks to Mumford-Metha-Ramanathan's restriction theorem \cite{MR}, it is possible to obtain the following Bogomolov-type result. Let $X$ be a variety of dimension $n$, and $\cE$ a vector bundle on $X$.
Let  $\cH$ be an ample line bundle on $S$. Define
$$\delta (\cE):= \deg \left(2\,\rk {\cal E} c_2({\cal E})H^{n-2}-(\rk {\cal E}-1)c_1^2({\cal E})H^{n-2}\right).$$
Then, if $\cE$ is  $\cH$-semistable, then $\delta(\cE)\geq 0$.
Unfortunately, this beautiful result does not imply $f$-positivity in dimension greater than $2$.
One should consult also the paper \cite{Mor2}, of Moriwaki himself, for other inequalities along the same lines.
\end{remark}




\section{Linear stability: a thread binding the methods for $n=2$}\label{thread}

In this section we introduce the {\em linear stability}, for curves together with a linear series.
We see that in the case of fibred surfaces linear stability represents a link between the three methods described in the previous section, that are from all other aspects extremely different.


Let us start by recalling the notion of linear stability for a curve and a linear series on it \cite{LS}.
This is a straightforward generalization in the case of curves of the one given by Mumford in \cite{Mum}.

Let  $C$ be a smooth curve, and let $\varphi\colon C\longrightarrow \pr^{r-1}$ be a non-degenerate morphism.
This corresponds to a  globally generated line bundle $\cL$  on $C$, and a base-point free linear subsystem $V\subseteq H^0(C,\cL)$ of dimension $r$ such that $\varphi$ is induced from the linear series $|V|$. Let  $d$ be the degree of $\mathcal L$ (i.e. $|V|$ is a a $g^{r-1}_d$ on $C$).
Linear stability gives a lower bound on the slope between the degree and the dimension of any projections, depending on the degree and dimension of the given linear series as follows.

\begin{definition}\label{ls}
With the above notation, we say that the couple  $(C, V)$, is {\em linearly semistable (resp. stable)} if  any linear series of degree $d'$ and dimension $r'-1$ contained in  $|V|$  satisfies $$\frac{d'}{r'-1}\geq \frac{d}{r-1}\quad (\mbox{resp. }>)$$

In case $V=H^0(\cL)$, we shall talk of the stability of the couple $(C, \cL)$. It is easy to see  that it is sufficient to verify that the inequality of the definition holds for any {\em complete} linear series in $|V|$.
\end{definition}

\begin{example}\label{linear stability known}
Some of the known results are the following.
\begin{enumerate}
\item The canonical system on a curve of genus $\geq 2$ is linearly semistable and it is stable if and only if the curve is non-hyperelliptic.
This follows from Clifford's Theorem and Riemann-Roch Theorem (see \cite[chap.14, sec.3]{ACG2}).
\item It is immediate to check that a plane curve of degree $d$ is linearly semistable (with respect to its immersion in $\pr^2$) if and only if it has points of multiplicity at most $d/2$.
\item Using Riemann-Roch Theorem, it is easy to check that the morphism induced on a curve of genus $g$ by a line bundle of degree $\geq 2g+1$ is linearly stable (see \cite{Mum}).
\item For a non-hyperelliptic curve of genus $\geq 2$, generic projections of low codimension from the canonical embedding  are linearly stable \cite{BS1}.
\item Given a base-point free linear system $V\subseteq H^0(C,\cL)$ on a curve $C$ of genus $g$,   if $\deg \cL \geq 2g$, and  the codimension of $V$ in $H^0(C, \cL)$ is less or equal than $(\deg \cL- 2g)/2$, then $(C, V) $ is linearly semistable \cite{mistretta}.
\end{enumerate}
\end{example}


\subsection{Linear stability and the Cornalba-Harris method}\label{LS&CH}

Mumford introduced the concept  of linear stability in order to find a more treatable notion that the ones of GIT stability.
The importance of linear stability from this point of view lies indeed in the following result \cite[Theorem 4.12]{Mum}

\begin{theorem}[Mumford]\label{LSimplicaCHOW}
If $(C, \cL)$ is linearly (semi-)stable and $\cL$ is very ample, then $(C, \cL)$ is Chow (semi-)stable.
\end{theorem}

In \cite{ACG2} it is proved the following result (\cite[Theorem (2.2)]{ACG2})
\begin{theorem}\label{LSimplicaHILBERT}
If $(C, \cL)$ is linearly stable and $\cL$ is very ample, then $(C, \cL)$ is Hilbert stable.
\end{theorem}

By Morrison's result  \cite[Corollary 3.5]{Mor}, we see that Theorem \ref{LSimplicaCHOW} implies Theorem \ref{LSimplicaHILBERT}.
The arguments of \cite[Theorem (2.2)]{ACG2} cannot be pushed through to the semistable case, so at present it is not known if linear strict semistability implies Hilbert strict semistabilty (through the authors would be surprised if it doesn't).

It is easy to extend  the proof of both Mumford's Theorem \ref{LSimplicaCHOW}, and \cite[Theorem (2.2)]{ACG2}  to the case of a very ample non necessarily complete linear system (see e.g. the second author's Ph.D. Thesis).
\begin{remark}
Let $V$ be a base-point free linear system on $C$  inducing a morphism $\varphi\colon C\longrightarrow \pr^r $ of positive degree on the image. It is immediate to see that the linear (semi)stability of $V$ is equivalent to the linear stability of the image $\varphi(C)$ with its embedding in  $\pr^r$ (compare with Remarks \ref{hilbfinito} and \ref{Chow-image}).
\end{remark}
Using  the results of Section \ref{C&H} we can thus state the following results.
\begin{theorem}\label{CHconLS}
Let $f\colon  S\longrightarrow B$ be a fibred surface, $\cL$  a line bundle on $S$ and $\cG\subseteq f_*\cL$ a subsheaf. 
Suppose that  for general $t\in B$ the couple $(F,  {G}_t)$  is linearly semistable.
Suppose moreover that we are in one of the following situations:
\begin{itemize}
\item[(i)] the couple $(F,  {G}_t)$ is strictly linearly stable, and the sheaf $\cG$ is either generating, or it satisfies the conditions of Proposition \ref{conto1,5} or  \ref{conto2};
\item[(ii)] $\,$ for $t\in B$ general, the fibre $G_t \subseteq H^0(F, \cL_{|F})$ is base-point free and  the line bundle $\cL$ is relatively nef.
\end{itemize}
Then the couple $(\cL, \cG)$ is $f$-positive (via the Cornalba-Harris method).
\end{theorem}


\subsection{Linear stability and Xiao's method}\label{LS&X}

We verify here that the method of  Xiao gives as a result the $f$-positivity under the assumption of linear stability.

\begin{theorem}\label{xiaoconLS}
Let $f\colon S\longrightarrow B$ a fibred surface, $F$ a general fibre, $\cL$ a nef line bundle on $S$ and $\cG\subseteq f_*\cL$ a nef rank $r$ subsheaf.
Assume that the linear system on $F$ induced by $\cG$ is linearly semistable.
Then $(\cL,\cG)$ is $f$-positive (via Xiao's method).
\end{theorem}

\begin{proof}

Following the description of Xiao's method given in Section \ref{XIAO}, consider the linear systems $P_i$ induced on $F$ by the pieces of the Harder-Narashiman filtration of $\cG$, of rank $r_i$.
Let $d_i={\rm deg}P_i$, and observe that $d_l={\rm deg}L_{|F}=:d$.
Linear stability condition implies
$$
\frac{d_i}{r_i-1}\geq \frac{d_l}{r_l-1}=\frac{d_l}{r-1}=:a\qquad \mbox{for any }i=1,\ldots, l.
$$
Observe that if $r_1=1$ then $d_1=0$ and the above inequality should be read as $d_1 \geq a r_1$ and still holds.
Consider now the sets of indexes $I=\{1,...,l\}$ and $I'=\{1,l\}$.
Then we have
$$
L^2 \geq \sum_{i=1}^{l}(d_i+d_{i+1})(\mu _i - \mu_{i+1})
$$
and
$$
L^2 \geq (d_1+d_l)(\mu _1 - \mu _l)+2d_l \mu_l\geq d_l(\mu_1+\mu_l).
$$
Use now that $d_i\geq a(r_i-1)$ for $i=1,...,l$ ($d_{l+1}=d_l$) and that $r_{i+1}\geq r_i+1$. Observe that ${\deg}\cG=\sum_{i=1}^{l}r_i(\mu_ i-\mu _{i+1})$ to get
$$
L^2 \geq 2a{\rm deg}\cG-a(\mu_1+\mu_l),
$$
which finally  proves
$$
L^2 \geq \frac{2ad_l}{a+d_l}{\rm deg}\cG=2\frac{d}{r}{\rm deg}\cG.
$$
\end{proof}

\begin{remark}
The fact that we used  Clifford's theorem in the proof of the slope inequality via Xiao's method in Example \ref{slopeX} can thus be rephrased in the following way: Clifford's theorem implies the linear semistability of the general fibres of $f$  together with their canonical systems.
\end{remark}
We can make the following improvement for the complete case.
\begin{proposition}
With the notations above, assume that $\cL$ is nef and that $\cL_{|F}$ is linearly semistable. Then $(\cL,f_*\cL)$ if $f$-positive, i.e.
$$L^2 \geq 2\frac{d}{r}{\rm deg}f_*\cL.$$
\end{proposition}

\begin{proof}Take $\cG$ to be the biggest piece of the Harder-Narashiman filtration of $f_*\cL$ such that $\mu_i\geq 0$. It is nef and we have that $\frac{d_i}{r_i-1}\geq \frac{d}{r-1}$ by linear semistability and that ${\rm deg}\cG \geq {\rm deg}f_*\cL$. Then apply the same method as in the proof of Theorem \ref{xiaoconLS}.
\end{proof}

\begin{remark}
From the proof of the Theorem \ref{xiaoconLS} we get an inequality even if we do not assume linear semistability condition on fibres. Indeed, observe that, if the linear subsystems of $P$, the one induced by $\cG$, verify
$$
\frac{d_i}{r_i-1} \geq a\qquad \mbox{for any }i=1,\ldots, l
$$
for some constant $a$, then we obtain the following inequality for the slope
$$L^2 \geq \frac{2ad}{a+d}{\rm deg}\cG.$$

\noindent  where $d={\rm deg}P$.

Take $\cG$ as in the proof of the previous proposition. Observe that $\frac{2ad_1}{a+d_1}\geq\frac{2ad_2}{a+d_2}$ if $d_1 \geq d_2$. Hence we can conclude that if $\cL$ is nef and induces a base point free linear system on $F$ of degree $d$, such that all its linear subsystems verify

$$\frac{d_i}{r_i-1} \geq a,$$

\noindent then

$$L^2 \geq \frac{2ad}{a+d}{\rm deg}f_*\cL.$$
\end{remark}

This remark allows us to give a general result for a  nef line bundle $\cL$ depending of its {\it degree of subcanonicity} (compare with \cite{Severi}).

\begin{proposition} Let $f: S \longrightarrow B$ be a fibred surface with general fibre $F$ of genus $g\geq 2$ and let $\cL$ be a nef line bundle on $S$. Let $d$ be the degree of the moving part of $\cL_{|F}$. Then

\begin{itemize}
\item [(i)] If $\cL_{|F}$ is subcanonical then $$L^2\geq \frac{4d}{d+2} {\rm deg}f_*\cL.$$
\item [(ii)] $\,$If $d\geq 2g+1$ then $$L^2\geq \frac{2d}{d-g+2} {\rm deg}f_*\cL.$$
\end{itemize}
\end{proposition}

\begin{proof}
\begin{itemize}
\item [(i)] Just take $a=2$ in the previous remark using Clifford's theorem.
\item [(ii)] $\,$ If $d \geq 2g+1$ then the linear system $\cL_{|L}$ is linearly semistable and hence we can take, by Riemann-Roch theorem on $F$,
$$a=\frac{d}{r-1}=\frac{d}{d+1-g}.$$
\end{itemize}
\end{proof}



\subsection{Linear stability and Moriwaki's method}\label{LS&M}

Let $C$ be a curve, $\cL$ a line bundle on $C$, and $V\subseteq
H^0(C, \cL)$ a linear subsystem of degree $d$ and dimension $r$.
We now compare the concept of linear stability for a couple $(C, V)$ with the stability needed for the application of Moriwaki's method.
We call
$$M_{\cL, V} := \ker( V\otimes \cO_C \longrightarrow \cL),$$
the {\em dual span bundle} (DSB) of the line bundle $\cL$ with respect to the generating subspace $V \subseteq H^0(C, \cL)$.
This is a vector bundle of rank $r-1$ and degree $-d$.
When $V = H^0(C, {\cal L})$ we denote it $M_{\cal L}$. \footnote{Note that we make here, as in \cite{MS}, an abuse of notation: properly speaking the dual span bundle is  the dual bundle of $M_{V,\cal L}$,
which is indeed spanned by $V^*$.}

\begin{remark}
An interesting  geometric interpretation of this sheaf is the following.
Consider the Euler sequence on $\pr^n$:
$$ 0 \longrightarrow  {\Omega}^{1}_{\pr^n} (1) \longrightarrow \cO_{\pr^n} ^{\oplus (n+1)}
\longrightarrow \cO_{\pr^n}(1) \longrightarrow 0.$$
Applying the pullback of $\varphi$ we obtain
$$0\longrightarrow \varphi^*({\Omega}^{1}_{\pr^n} (1) )\longrightarrow V\otimes \cO_C\longrightarrow \cL\longrightarrow 0.$$
Hence the kernel of the evaluation morphism coincides with
the restriction of the tangent bundle of the projective space $\pr^n$ to the curve $C$.
\end{remark}


The $\mu$-stability of the DSB is the stability condition assumed to hold on the general fibres for the method of Moriwaki.

\begin{proposition}\label{implicazione facile}
With the above notation, if the DSB sheaf  $M_{\cL, V}$ is $\mu$-(semi)stable, then the couple $(C,V)$ is linearly (semi)stable.
\end{proposition}
\begin{proof}
Let  us consider any  $g_{d'}^{r'-1}$ in  $|V|$. Let $V'$ be the associated subspace of $V$.
Consider the evaluation morphism $V'\otimes \cO_C\longrightarrow \cL$, which is not surjective unless $d'=d$, and let $\cG$ be its kernel. Then $\cG$ is a vector subbundle of $M_{\cL, V}$ with $\deg \cG=-d'$, $\rk \cG=r'-1$. So from the stability condition on  $M_{\cL, V}$ we obtain that $d'/(r'-1)\geq d/(r-1)$.
\end{proof}
\begin{remark}\label{la cosa bella}
Note that any $\cG\subseteq M_{\cL, V}$ as in the above theorem fits into the diagram
$$
\xymatrix{
0\ar[r]&\cG\ar[r]\ar@{^{(}->}[d]&V'\otimes \cO_C\ar[r] \ar@{^{(}->}[d] &\overline{\cL}\ar[r]\ar@{^{(}->}[d]& 0\\
0\ar[r]&M_{\cL, V}\ar[r]&V\otimes \cO_C\ar[r]  &\cL\ar[r]& 0}
$$
where $\overline \cL\subseteq \cL$ is the image of the evaluation morphism $V'\otimes \cO_C\longrightarrow \cL$.
\end{remark}

The converse implication is studied in \cite{MS}.
Let us now briefly discuss the question.
Consider a linearly semistable couple $(C, V)$.
Let us consider a proper saturated subsheaf $\cG\subseteq M_{\cL, V}$.
We have that $\cG^*$ is generated by its global sections. Consider the image of the natural morphism $W^*:=\mbox{im}(V^*\rightarrow H^0( C, \cG^*))$.
The evaluation morphism $W^*\otimes \cO_C\longrightarrow \cG^*$ is surjective.
We thus have the following commutative diagram (cf. \cite{Butler})
\begin{equation}\label{booo}
\xymatrix{
0\ar[r]&\cG\ar[r]\ar@{^{(}->}[d]&W\otimes \cO_C\ar[r] \ar@{^{(}->}[d] &\cF\ar[r]\ar[d]_\alpha& 0\\
0\ar[r]&M_{\cL, V}\ar[r]&V\otimes \cO_C\ar[r]  &\cL\ar[r]& 0}
\end{equation}
where $\cF$ is a vector bundle without trivial summands (this follows from the choice of $W$). Note that the  morphism $\alpha$ is non-zero, because $W\otimes \cO_C$ is not contained in the image of $M_{\cL, V}$ in $V\otimes \cO_C$.
Let us suppose that $\cG$ is destabilizing: if the sheaf $\cF$ is a line bundle, then we would be in the situation described in Remark \ref{la cosa bella}, and we could easily deduce from the linear (semi)stability of $(C, V)$ the $\mu$-(semi)stability of $M_{\cL, V}$.
But we cannot exclude that a destabilizing subsheaf exists which is not the transform of a line bundle contained in $\cL$.

In \cite{MS} the second author and E. C. Mistretta prove that indeed this is the case in the  following cases, which depend on the Clifford index $\textrm{Cliff} (C)$ of the curve $C$.

\begin{theorem}[\cite{MS} Theorem 1.1]
Let ${\cal L} $ be a globally generated line bundle on $C$,
and $V \subseteq H^0(C, \cal L)$ a generating space of global sections
such that
$$
\deg {{\cal L}} - 2 (\dim V  -1) \leqslant \textrm{Cliff} (C).
$$
Then linear (semi)stability of $(C, V)$
is equivalent to $\mu$-(semi)stability
of $M_{V,{\cal L}}$ in the following cases:

\begin{enumerate}
\item $ V = H^0({\cal L}) $;
\item $\deg {\cal L}  \leqslant 2g - \textrm{Cliff} (C) +1$;
\item $\mathrm{codim}_{H^0({\cal L})} V <h^1({\cal L})+g/(\dim V-2)$;
\item  $\deg {\cal L} \geqslant 2g$,
and $ \mathrm{codim}_{H^0({\cal L})} V \leqslant (\deg {\cal L} -2g)/2 $.
\end{enumerate}
\end{theorem}




\section{Results in higher dimensions}\label{resultsHIGH}
\subsection{Linear stability in higher dimensions and Xiao's method}
Mumford's original definition of linear stability is in any dimension, as follows.
\begin{definition}[\cite{Mum}, Definition 2.16]
An $m$-dimensional variety of degree $d$ in $\pr^{r-1}$ is linearly semistable (resp. linearly stable) if for any projection $\pi\colon \pr^{r-1}\dasharrow \pr^{s-1}$ such that the image of $X$ is still of dimension $m$, the following inequality holds:
$$
\frac{\deg(\pi_*(X))}{s-m}\geq \frac{d}{r-m} \qquad (resp. \, >),
$$
where $\pi_*(X)$ denotes the image cycle of $X$ in $\pr^s$.
\end{definition}
For example, it easy to verify that a K3 surface with Picard number $1$ is linearly semistable.
However, as Mumford himself remarks some lines after the definition, this condition in dimension higher than $1$ seems to be difficult to handle.
Moreover, it does not imply anymore  Hilbert or Chow stability.

It seems  that  there is no sensible connection between linear stability and the method of Cornalba-Harris.

The relation of linear semistability with $f$-positivity via Xiao's method appears clear enough when all the induced maps of the Harder-Narashiman pieces are generically finite onto its image. More concretely, we obtain an inequality very close to $f$-positivity (it is possible to get something slightly better with much more effort).

\begin{proposition}\label{Xiao-high}
 Let $f:X \longrightarrow B$ be a fibration with general fibre $F$, $n={\rm dim}X$ and $\cL$ a nef line bundle. Let $\cG \subseteq f_*\cL$ be a nef subbundle. Assume that all the induced maps on $F$ by the Harder-Narashimann pieces of $\cG$ are generically finite and that the one induced by $\cG$ is (Mumford-)linearly semistable. Then
$$L^n  \geq n\frac{d}{r+(n-1)^2}{\rm deg}\cG.$$

\end{proposition}

\begin{proof} A similar argument as in Theorem \ref{xiaoconLS} applies. Let $a=d/(r-n+1)$, where $d$ and $r$ are the degree and rank of the base-point free linear map induced on (a suitable blow-up of) $F$ by $\cG$. By Xiao's inequality, using that all the induced maps on fibres are generically finite onto their images and that  $P_{i+1}^kP_i^{r-1-k} \geq P_i^{r-1}$for all $i$, we obtain

$$L^n \geq n(\sum_i P_i^{n-1}(\mu_i-\mu_{i+1}))\geq \sum_i (nar_i-n(n-1)a)(\mu_i-\mu_{i+1})=na{\rm deg}\cG-n(n-1)a\mu_1.$$

\noindent Since $L-\mu_1F$ is pseudoeffective and $L$ is nef we have that $L^{n-1}(L-\mu_1F)\geq 0$ and so $$L^n\geq \mu_1d,$$
\noindent which finally gives

$$L^n \geq \frac{nad}{a+d}{\rm deg}\cG \geq n\frac{d}{r+(n-1)^2}{\rm deg}\cG.$$
\end{proof}

Clearly, the argument above does not work in general for dim$X \geq 3$, due to the presence of induced map on fibres which are not generically finite. In some situations, however, it is possible to control such maps and conclude again $f$-positivity. In \cite{BS-K3} we do this analysis for families of $K3$ surfaces, obtaining a significant generalization of Proposition \ref{slopeK3} below.





\subsection{New inequalities and conjectures via the C-H method}\label{CH-high}
We now state a couple of  new results obtained via the CH method, using known stability results in dimension $\geq 2$, and make some speculation and natural conjectures.

\subsubsection*{Families of abelian varieties}
Let us consider a fibred variety $f\colon X\longrightarrow B$ of dimension $n$ such that the general fibre is an abelian variety.
Suppose that $\cL$ is a line bundle on $X$ such that $ f_*\cL$  is either generating, or it satisfies the conditions of Proposition \ref{conto1,5} or  \ref{conto2}.
\begin{proposition}
Under the above assumption, suppose that $\cL$ is very ample on the general fibre. Then  $(\cL,\cG)$ is $f$-positive,
i.e.
\begin{equation}\label{slopeABELIAN}
L^n \geq n! \deg f_*\cL.
\end{equation}
\end{proposition}
\begin{proof}
We can apply Theorem \ref{corCH} because the immersion induced by $\cL_{|F}$ on the general fibre $F$  is Hilbert semistable by Kempf's result \cite{kempf}. Observe then that as $F$ is abelian, and $\cL$ is very ample, we have that $h^0(F, \cL_{|F})=\chi (\cL_{|F})= L_{|F}^{n-1}/(n-1)!$, and so $f$-positivity translates in formula (\ref{slopeABELIAN}).
\end{proof}

\subsubsection*{Families of K3 surfaces}
Let $f\colon T\longrightarrow B$ be a fibred threefold such that the general fibre is a $K3$ surface of genus $g$.
Let $\cL$ be a line bundle on $T$ such that $f_*\cL$ is either generating, or it satisfies the conditions of Proposition \ref{conto1,5} or  \ref{conto2}.
The following result follows right away from  Theorem \ref{corCH} applying Morrison's result \cite{morrison-K3}.
\begin{proposition}\label{slopeK3}
In the above situation, suppose that the general fibres $F$ have Picard number $1$, that $\cL_{|F}$ is the primitive divisor class, and that its degree is at least $12$.
Then $\cL$ is $f$-positive, {\em i.e.} the following inequality holds:
$$
L^3\geq 6\frac{g-1}{g+1}\deg f_*\cL.
$$
\end{proposition}

\begin{remark}
It is interesting to notice that the bound $6(g-1)/(g+1)$ appearing in the inequality of Proposition \ref{slopeK3} coincides with the one obtained in \cite{konnocliff} and in \cite{HE} for the canonical slope of  fibred surfaces of odd degree $g$ whose general fibre is of maximal gonality.

Moreover, it is easy to prove that the canonical slope of a family of curves contained in a {\em fixed} $K3$ surface is indeed bounded from below by $6(g-1)/(g+1)$.
\end{remark}

\subsubsection*{A conjecture on the slope inequality in higher dimension}

Let $f\colon X\longrightarrow B$  be an $n$-dimensional fibred variety.
It is natural to define as a possible canonical slope the ratio between $K_f^n$ and $\deg f_*\omega_f$, another possibility being to use the relative characteristic $\chi_f$: in higher dimension the values  $\deg f_*\omega_f$ and $\chi_f$ are not equal, but it holds an inequality between them \cite{ohno} \cite{Barja3folds}.

A natural slope inequality in higher dimension would be the following
\begin{equation}\label{ineq-slope-high}
K_f^n\geq n\frac{K_F^{n-1}}{h^0(F, \omega_F)}\deg f_*\omega_f,
\end{equation}
which is equivalent to the $f$-positivity of $\omega_f$.
From the Cornalba-Harris and Bost method we can derive inequality (\ref{ineq-slope-high}) any time we have a Hilbert-Chow semistable canonical map on the general fibres.
Although there are not much general results, it seems natural in the framework of GIT to conjecture that the stability of a variety has a connection with its singularities: a stable or asymptotically stable variety has mild singularities and it seems that also a vice-versa to this statement should hold.
In consideration of this fact, and in analogy with the case of curves, it seems natural to state the following conjecture. See also Remark \ref{slopeinequality} for an account the natural positivity conditions on $\omega_f$.
\begin{conjecture}\label{conj-slope}
Let   $f\colon X\longrightarrow B$  be a fibred $n$-dimensional  variety whose relative canonical sheaf $\omega_f$ is relatively nef and ample on the general fibres, and whose general fibres have sufficiently mild singularities (e.g. they are log canonical, or semi-log-canonical).
Then the fibration  satisfies the slope inequality (\ref{ineq-slope-high}).
\end{conjecture}
Almost nothing is known about this conjecture in dimension higher than $2$. In \cite{hyper} we prove this inequality for families of hypersurfaces whose general fibres satisfy a very weak singularity condition expressed in terms of its log canonical threshold and depending upon the degree of the hypersurfaces (see \cite{Lee}).

\begin{remark}
Recall that the Severi inequality for surfaces $S$ of maximal Albanese dimension $
K_X^2\geq 4 \chi(\cO_X)
$
has been proved in full generality by Pardini in \cite{Parda}.
In \cite{Severi} the first author proves that higher dimensional Severi inequalities of the form $L^n\geq 2n!\chi (\cL)$ hold n arbitrary dimensions for any nef line bundle $\cL$.
The classical proof of Severi inequality for surfaces and $\cL=\omega_S$ given by Pardini makes use of the  slope inequality for fibred surfaces.
We prove now that her argument can be generalized, assuming that  Conjecture \ref{conj-slope} holds.
\end{remark}

\begin{proposition}\label{slope-severi}
Let $m>0$ be an integer.
Suppose that slope inequality (\ref{ineq-slope-high}) holds for all varieties of dimension $\leq m$ that have  maximal Albanese dimension and are fibred over  $\mathbb{P}^1$.
Then for any variety $X$ of dimension $n\leq m$ with maximal Albanese dimension it holds the following sharp Severi inequality:
\begin{equation}\label{S}
K_X^n\geq 2n! \chi(\omega _X).
\end{equation}
\end{proposition}
\begin{proof} We proceed by induction on $n={\rm dim}X$. For $n=1$ inequality (\ref{S}) is trivially true.
Take now $n\geq 2$. First of all observe that from the slope inequality we can deduce a stronger result for maximal Albanese dimensional varieties with fibrations $f\colon X\longrightarrow \pr^1$. Indeed, consider an \'etale Galois cover of $X$ of degree $r$, and the induced fibration ${\widetilde f}$.
Then, applying inequality (\ref{ineq-slope-high}) we obtain
$$
rK_f^n=K_{\widetilde f}^n \geq n\frac{rK_F^{n-1}}{r\chi(\omega_F)+\epsilon_1}(r\chi_f+\epsilon_2),
$$

\noindent where $\epsilon _1=(h^1(F,\omega_F)-...+(-1)^{n-2}h^{n-1}(F,\omega_F))$ and  $\epsilon _2=({\rm deg}R^1f_*\omega_f-...$ $...+(-1)^{n-2}R^{n-1}f_*\omega_f)$. Since the inequality holds for all $r$ we obtain
$$
K_F^n+2nK_F^{n-1}=K_f^n \geq n\frac{K_F^{n-1}}{\chi(\omega_F)}(\chi (\omega_X)+2\chi(\omega_F)).
$$
Applying induction hypothesis for $F$ (which is clearly of maximal Albanese dimension), we deduce the inequality
\begin{equation}\label{slopemejor}
K_F^n+2nK_F^{n-1}\geq 2n!(\chi (\omega_X)+2\chi(\omega_F)).
\end{equation}
Now we can ``eliminate the contribution due to $F$'' just mimetizing Pardini's argument in \cite{Parda}, which we sketch here.

Consider the following cartesian diagram
$$
\xymatrix { {\widetilde X} \ar[r]^{\mu} \ar[d]^{\widetilde a} & X \ar[d]^{a} \\ A \ar[r]^{\mu} & A}
$$
where $a\colon X\longrightarrow A$ is the Albanese map, and the maps $\mu$ are multiplication by $d$ in $A$ and so are Galois \'etale maps of degree $d^{2q}$. Fix a very ample line bundle $\cH$ on $A$ and let $\cM=a^*(\cH)$ and ${\widetilde \cM}={\widetilde a}^*(\cH)$.
By \cite[Ch2. Prop.3.5]{BirkLange} we have that
$$
{\widetilde M} \equiv \frac{1}{d^2}\mu ^*(M)\quad \mbox{ (numerical equivalence).}
$$

 Take general elements $F,F' \in |{\widetilde M}|$ and perform a blow-up $Y\longrightarrow X$  to obtain a fibration $f: Y \longrightarrow \mathbb{P}^1$. Then we apply (\ref{slopemejor}) to $f$ and obtain
$$
K^n_Y+2nK_F^{n-1} \geq 2n!(\chi( \omega_Y)+2\chi(\omega_F)).
$$
Now an easy computation through the blow-up and the \'etale cover $\mu$ shows that
\begin{itemize}
\item $K^n_Y=d^{2q}K_X^n+\cO(d^{2q-4})$.
\item $K_F^{n-1}=K_{\widetilde X}^{n-1}{\widetilde M}+(n-1)K_{\widetilde X}^{n-2}{\widetilde M}=\cO(d^{2q-2})$.
\item $\chi (\omega_Y)=\chi (\omega_{\widetilde X})=d^{2q}\chi( \omega_X)$.
\item $\chi (\omega_F)=\cO(d^{2q-2})$ by Riemann-Roch theorem on ${\widetilde X}$.
\end{itemize}
Since these equalities holds for any $d$ we conclude that
$$
K_X^n \geq 2n! \chi(\omega_X).
$$

\end{proof}







\end{document}